%% file: coconuts-v2.tex
\documentclass[12pt]{article}  

\usepackage[left=1in,right=1in,top=1in,bottom=1in,footskip=.4in]{geometry}
\usepackage[ruled,lined,linesnumbered,vlined]{algorithm2e}
\usepackage{algorithmic}
\usepackage{amsfonts}
\usepackage{amssymb}
\usepackage{color}
\usepackage{amsmath}
\usepackage{graphicx}
\usepackage{multirow}
\usepackage{hyperref}
\usepackage{tikz}

\usepackage{definitions}
\usepackage{arXivOnlyDefs}
\newcommand{\coap}{}
\newcommand{\paoc}{\nonumber\\&\quad}
\newcommand{\coapnquad}{}
\newcommand{\coapAmp}{}
\newcommand{\coapDD}{}
\newcommand{\aCoapLine}{}

\newcommand{\ourqed}{}

\newcommand{\coapversion}[1]{}
\newcommand{\preprintversion}[1]{#1}

\newcommand{\myc}[1]{#1}

\title{Single-Forward-Step Projective Splitting: Exploiting Cocoercivity}
\author{Patrick R. Johnstone\thanks{Department of Management Science 
		and Information Systems, Rutgers Business
		School Newark and New Brunswick, Rutgers University. Contact: \href{mailto::patrick.r.johnstone@gmail.com}{patrick.r.johnstone@gmail.com}, \href{mailto::jeckstei@business.rutgers.edu}{jeckstei@business.rutgers.edu}}
	\and 
	Jonathan Eckstein$^*$}

\begin{document}  
 
\maketitle

\input{abstract}

\input{introduction}

\input{introHistory}

\input{exampleIntuition}

\input{introProposedUpdate}
\input{onePGstep}

\input{mainAssump}

\input{theHyperplane2}
\input{genericUp}
\input{allAlgDefs}

\input{theAlgorithm}

\input{stepAssump}

\input{sepProperties}

\input{nIsOne}
\input{mainProof}

\input{contractive}

\input{ascent}

\input{bounded}

\input{mainProof-part2}

\input{portfolio}

\input{group-lasso}

\input{onefVtwof}

\subsection*{Acknowledgments}

This research was supported by the National Science Foundation grant CCF-1617617.

\bibliographystyle{spmpsci}
\bibliography{refs}


\end{document}

%% file: abstract.tex
\begin{abstract} 
This work describes a new variant of projective splitting for solving \myc{maximal} monotone 
inclusions \myc{and complicated convex optimization problems}. \myc{In the new
version}, cocoercive
operators can be processed with a single forward step per iteration. 
\myc{ In the convex optimization context, cocoercivity is equivalent to Lipschitz
differentiability.}
\myc{Prior forward-step versions of projective splitting did not fully exploit 
cocoercivity and required \emph{two} forward steps per iteration for such operators. 
Our new single-forward-step method} establishes a symmetry between projective
splitting algorithms, the classical forward-backward splitting method (FB),
and Tseng's forward-backward-forward method (FBF).
The new procedure allows for larger stepsizes for cocoercive
operators: the stepsize bound is $2\beta$ for a $\beta$-cocoercive operator,
\myc{the same bound as has been established} for FB. We show that FB
corresponds to an unattainable boundary case of the parameters in the new
procedure. Unlike FB, the new method allows for a backtracking procedure when
the cocoercivity constant is unknown.
Proving convergence of the algorithm requires some departures from the \myc{prior}
proof framework for projective splitting.
We close with some computational tests establishing competitive performance 
for the method. 
\coapversion{\keywords{proximal operator splitting \and projective splitting \and convex nonsmooth optimization}}
\end{abstract}

%% file: introduction.tex
\section{Introduction} 
\subsection{Problem Statement}
For a collection of real Hilbert spaces $\{\calH_i\}_{i=0}^n$ consider the \emph{finite-sum convex minimization problem}:
\begin{eqnarray}\label{ProbOpt}
\min_{x\in\mathcal{H}_{0}} \sum_{i=1}^{n}\big(f_i(G_i x)+h_i(G_i x)\big),
\end{eqnarray}
where every $f_i:\calH_i\to(-\infty,+\infty]$ and $h_i:\calH_i\to\rR$ is closed, proper, and convex, every $h_i$ 
is also differentiable with $L_i$-Lipschitz-continuous gradients, and the operators $G_i:\calH_0\to\calH_i$ are linear and bounded. 
Under appropriate constraint qualifications, \eqref{ProbOpt} is equivalent to the  monotone inclusion problem of finding $z\in\calH_0$ such that
\begin{align}\label{prob1}
0\in \sum_{i=1}^n G_i^*\left(A_i  + B_i\right) G_iz
\end{align} 
where all $A_i:\calH_i\to2^{\calH_i}$ and $B_i:\calH_i\to\calH_i$ are maximal monotone and each  $B_i$ is $L_i^{-1}$-cocoercive, meaning that it is single-valued and
\begin{align}\nonumber 
L_i \langle B_i x_1 - B_i x_2,x_1-x_2\rangle\geq \|B_i x_1-B_i x_2\|^2
\end{align}
for some $L_i\geq 0$. (When $L_i=0$, $B_i$ must be a constant operator, that
is, there is some $v_i\in\calH_i$ such that $B_i x = v_i$ for all
$x\in\calH_i$. ) In particular, if we set $A_i = \partial f_i$ \myc{(the
subgradient map of $f_i$)} and $B_i = \nabla h_i$  \myc{(the gradient of
$h_i$)} then the solution sets of the two problems coincide under a special
case of the constraint qualification of
\cite[Prop.~5.3]{combettes2013systems}.

Defining $T_i
= A_i + B_i$ for all $i$, problem~\eqref{prob1} may be written as
\begin{align}
\label{prob2}
0\in \sum_{i=1}^n G_i^*T_i G_iz.
\end{align}
\myc{This more compact problem statement will be used occasionally in our analysis
below}.

%% file: introHistory.tex
\myc{ 
\subsection{Background}
Operator splitting algorithms are an effective way to solve structured convex
optimization problems and monotone inclusions such as \eqref{ProbOpt},
\eqref{prob1}, and \eqref{prob2}. Their defining feature is that they decompose
a problem into a set of manageable pieces.  Each iteration consists of
relatively easy calculations confined to each individual component of the
decomposition, in conjunction with some simple coordination operations 
orchestrated to converge to a solution.
Arguably the three most popular classes of operator splitting algorithms
are the forward-backward splitting (FB) \cite{combettes2011proximal},
Douglas/Peaceman-Rachford splitting (DR) \cite{lions1979splitting}, and
forward-backward-forward (FBF) \cite{tseng2000modified} methods. 
Indeed, many
algorithms in convex optimization and monotone inclusions are in fact
instances of one of these methods. The popular Alternating Direction Method of
Multipliers (ADMM), in its standard form, can be viewed as a dual
implementation of DR \cite{gabay1983chapter}.

Projective splitting is a relatively recent and currently less well-known
class of operator splitting methods, operating in a primal-dual space.  Each
iteration $k$ of these methods explicitly contructs an affine ``separator''
function $\varphi_k$ for which $\varphi_k(p) \leq 0$ for every $p$ in the set
$\calS$ of primal-dual solutions.  The next iterate $p^{k+1}$ is then obtained
by projecting the current iterate $p^k$ onto the halfspace defined by
$\varphi_k(p) \leq 0$, possibly with some over- or under-relaxation.
Crucially, $\varphi_k$ is obtained by performing calculations that consider
each operator $T_i$ separately, so that the procedures are indeed operator
splitting algorithms.  In the original formulations of projective
splitting~\cite{eckstein2008family,eckstein2009general}, the calculation
applied to each operator $T_i$ was a standard resolvent operation, also known
as a ``backward step''.  Resolvent operations remained the only way to process
individual operators as projective splitting was generalized to cover
compositions of maximal monotone operators with bounded linear
maps~\cite{alotaibi2014solving} --- as in the $G_i$ in~\eqref{prob2} --- and
block-iterative (incremental) or asynchronous calculation
patterns~\cite{combettes2016async,eckstein2017simplified}. Convergence rate
and other theoretical results regarding projective splitting may be found in
\cite{johnstone2018projective2,johnstone2018convergence,machado2017complexity,machado2016projective}.}

%
%


\myc{The algorithms in~\cite{tran2015new,johnstone2018projective} were the first to
construct projective splitting separators by applying calculations other than
resolvent steps to the operators $T_i$.  In particular,
\cite{johnstone2018projective} developed a procedure that could instead use
two \emph{forward} (explicit or gradient) steps for operators $T_i$ that are
Lipschitz continuous.
However, that} result raised a
question: if projective splitting can exploit Lipschitz continuity, can it
further exploit the presence of \emph{cocoercive} operators? Cocoercivity is
in general a stronger property than Lipschitz continuity. However, when an
operator is the gradient of a closed proper convex function (such as
$h_i$ in \eqref{ProbOpt}), the Baillon-Haddad theorem
\cite{baillon1977quelques,bauschke2009baillon} establishes that the two
properties are equivalent: $\nabla h_i$ is $L_i$-Lipschitz
continuous if and only if it is $L_i^{-1}$-cocoercive.

Operator splitting methods that exploit cocoercivity rather than mere
Lipschitz continuity typically have lower per-iteration computational
complexity and a larger range of permissible stepsizes. For example, both 
FBF 
and the extragradient (EG) method~\cite{korpelevich1977extragradient} only
require Lipchitz continuity, but need \emph{two} forward steps per iteration
and limit the stepsize to $L^{-1}$, where $L$ is the Lipschitz constant. If one
strengthens the assumption to $L^{-1}$-cocoercivity, one can instead use
FB, which only
needs one forward step per iteration and allows stepsizes bounded away from
$2L^{-1}$. One departure from this pattern is the recently developed
method of \cite{tam2018forward}, which only requires Lipschitz continuity but
uses just one forward step per iteration. While this property is remarkable,
it should be noted that its stepsizes must be bounded by $(1/2)L^{-1}$,
which is half the allowable stepsize for EG or FBF \myc{and just a fourth of FB's stepsize range}.

Much like EG and FBF, the projective splitting computation in
\cite{johnstone2018projective} requires Lipschitz continuity\footnote{If
backtracking is used, then all three of these methods can converge under weaker
local continuity assumptions.}, two forward steps per iteration, and limits
the stepsize to be less than $L^{-1}$ (when not using backtracking).
Considering the relationship between FB and FBF/EG leads to the following question:
is there a variant of projective splitting which converges under the stronger
assumption of $L^{-1}$-cocoercivity, while processing each cocoercive operator
with a single forward step per iteration and allowing stepsizes bounded above by
$2L^{-1}$?

\myc{This paper shows that the answer to this question is ``yes''.}
Referring to
\eqref{prob1}, the new procedure \myc{analyzed here} 
requires one forward step on~$B_i$ and one
resolvent for~$A_i$ at each iteration.
\myc{In the context of \eqref{ProbOpt}, the new procedure requires one forward step
on $\nabla h_i$ and one proximal operator evaluation on $f_i$.} When the
resolvent is easily computable (for example, when $A_i$ is the zero map and
its resolvent is simply the identity), the new procedure can effectively halve
the computation necessary to run the same number of iterations as the previous
procedure of \cite{johnstone2018projective}. This advantage is equivalent to
that of FB over FBF and EG when cocoercivity is present. Another advantage of
the proposed method is that it allows for a backtracking linesearch when the
cocoercivity constant is unknown, whereas \myc{no such variant of general 
cocoercive FB is currently known}.


The analysis of this new method is significantly different from our previous
work in \cite{johnstone2018projective}, using a novel ``ascent lemma'' (Lemma
\ref{lemAscent1}) regarding the \myc{separators} generated by the
algorithm.
%
%
The new procedure also
has an interesting connection to the original resolvent
calculation used in the projective splitting papers
\cite{eckstein2008family,eckstein2009general,alotaibi2014solving,combettes2016async}: in Section \ref{secConnect2Prox}
below, we show that the new procedure is equivalent to one iteration of FB
applied to evaluating the resolvent of~$T_i = A_i+B_i$.  That is, we can use
a single forward-backward step to approximate the operator-processing procedure 
of~\cite{eckstein2008family,eckstein2009general,alotaibi2014solving,combettes2016async}, but still obtain convergence.

The new procedure has significant potential for asynchronous and
incremental implementation following the ideas and techniques of previous
projective splitting methods
\cite{combettes2016async,eckstein2017simplified,johnstone2018projective}.
To keep the analysis \myc{relatively} manageable, however, we plan to develop such
generalizations in a follow-up paper.  Here, we will simply assume that every
operator is processed once per iteration.


 



\myc{
\subsection{The Optimization Context}\label{secOptCon}


For optimization problems of the form~\eqref{ProbOpt}, our proposed method is
a first-order proximal splitting method that ``fully splits'' the problem: at
each iteration, it utilizes the proximal operator for each nonsmooth function
$f_i$, a single evaluation of the gradient $\nabla h_i$ for each smooth
function $h_i$, and matrix-vector multiplications involving $G_i$ and $G_i^*$.
There is no need for any form of matrix inversion, nor to use resolvents of
composed functions like $f_i \circ G_i$, which may in general be much more
challenging to evaluate than resolvents of the $f_i$.  Thus, the method
achieves the maximum possible decoupling of the elements of~\eqref{ProbOpt}.
There are also no assumptions on the rank, row spaces, or columns spaces of
the $G_i$. Beyond the basic resolvent, gradient, and matrix-vector
multiplication operations invoked by our algorithm, the only computations at
each iteration are a constant number of inner products, norms, scalar
multiplications, and vector additions, all of which can all be carried out
within flop counts linear in the dimension of each Hilbert space. }


\myc{Besides projective splitting approaches, there
	are a few first-order proximal splitting methods that} can achieve full
	splitting on~\eqref{ProbOpt}. The most similar to projective splitting are
	those in the family of \emph{primal-dual} (PD) splitting methods; see
	\cite{condat2013primal,combettes2012primal,chambolle2011first,PesquetAudrey2015}
	and references therein. In fact, projective splitting is also a kind of
	primal-dual method, since it produces \myc{primal and dual
	sequences jointly} converging to a primal-dual solution.  However, the
	convergence mechanisms are different: PD methods are usually constructed
	by applying an established operator splitting technique such as FB, FBF,
	or DR to an appropriately formulated primal-dual inclusion in a
	primal-dual product space, \myc{possibly with a specially chosen metric.
	Projective splitting methods instead work by projecting onto (or through)
	explicitly constructed separating hyperplanes in the primal-dual space.}

	There are \myc{several} potential
	advantages of our proposed method over the more established PD schemes.
	First, unlike the PD methods, the norms $\|G_i\|$ do not effect the
	stepsize constraints of our proposed method, making
	such constraints easier to satisfy. Furthermore, 
	\myc{projective splitting's}
	stepsizes may vary at each iteration and may differ for each operator.
	\myc{In general, 
		} 
	projective splitting methods allow for asynchronous parallel and
	incremental implementations in an arguably simpler way than PD methods
	(although we do not develop this aspect 
	\myc{of projective splitting} in this paper). 
	\myc{Projective splitting methods can incorporate \emph{deterministic}
	block-iterative and 
	asynchronous assumptions \cite{combettes2016async,eckstein2017simplified},
	resulting in deterministic convergence guarantees}, with the analysis being 
	similar to the synchronous case. In contrast, existing asynchronous and
	block-coordinate analyses of PD methods require
	stochastic assumptions which only lead to probabilistic convergence
	guarantees~\cite{PesquetAudrey2015}.
	
 \preprintversion{
 	\subsection{Notation and a Simplifying Assumption}
 	We use the same general notation as in
 	\cite{johnstone2018projective,johnstone2018convergence,johnstone2018projective2}.
 	Summations of the form $\sum_{i=1}^{n-1}a_i$ 
 	will appear throughout this paper. To deal with the case $n=1$, we use the
 	standard convention that
 	$
 	\sum_{i=1}^{0} a_i = 0.
 	$
 	
 	We will use a boldface $\bw = (w_1,\ldots,w_{n-1})$ for elements of
 	$\calH_1\times\ldots\times\calH_{n-1}$. Let $\boldsymbol{\mathcal{H}}
 	\triangleq \mathcal{H}_0\times
 	\mathcal{H}_1\times\cdots\times\mathcal{H}_{n-1}$, which we refer to as the
 	``collective primal-dual space'', and note that the assumption on $G_n$
 	implies that $\calH_n = \calH_0$. We use $p$ to refer to points in $\bcalH$,
 	so $p\triangleq (z,\bw)=(z,w_1,\ldots,w_{n-1})$.
 	
 	Throughout, we will simply write $\|\cdot\|_i = \|\cdot\|$ as the
 	norm for $\calH_i$ and let the subscript be inferred from the argument. In the
 	same way, we will write $\langle\cdot,\cdot\rangle_i$ as $\langle\cdot
 	,\cdot\rangle$ for the inner product of $\calH_i$. For the collective
 	primal-dual space we will use a special norm
 	and inner product with its own subscript defined in \eqref{gammanorm}.
 }

\preprintversion{
	We use the standard ``$\rightharpoonup$" notation to denote weak convergence,
	which is of course equivalent to ordinary convergence in finite-dimensional
	settings.
}

\coapversion{

\subsection{Notation and a Simplifying Assumption}	

}

\myc{ 
For the definition of maximal monotone operators and their basic properties, we refer to \cite{bauschke2011convex}.	
}	
For any maximal monotone operator $A$ \myc{and scalar $\rho > 0$,} we will use the notation
 $
 J_{\rho A} \triangleq (I+\rho A)^{-1},
 $
 to denote the \emph{resolvent operator}, 
 also known as the backward or implicit step with respect to $A$.  Thus,
 \begin{eqnarray}\label{defprox2}
 x = J_{\rho A}(t) \quad \iff \quad x+\rho a = t \;\;\text{and}\;\; a\in Ax,
 \end{eqnarray}
 \myc{the $x$ and $a$ satisfying this relation being unique.}
 Furthermore, $J_{\rho A}$ is defined everywhere and
 $\text{range}(J_A) = \text{dom}(A)$ \cite[Prop. 23.2]{bauschke2011convex}. 
 \myc{ 
If $A = \partial f$ for a closed, convex, and proper function $f$, the resolvent is often referred to as the \emph{proximal operator} and written as $J_{\rho\partial f}=\prox_{\rho f}$. Computing the proximal operator requires solving
\begin{align*}
\prox_{\rho f}(t) = \argmin_{z}\left\{
									\rho f(z)+\frac{1}{2}\|z-t\|^2
  						       \right\}.
\end{align*}
Many functions encountered in applications to machine learning and signal
processing have proximal operators which can be computed exactly with low
computational complexity. In this paper, for a single-valued maximal monotone
operator $A$, a \emph{forward step} (also known as an explicit step) refers to
the direct evaluation of $A x$ (or $\nabla f(x)$ in convex optimization) as
part of an algorithm. 
}

 For the rest of the paper, we will impose the simplifying assumption
 \begin{align}\nonumber
 G_{n}:\mathcal{H}_{n}\to\mathcal{H}_{n} &\triangleq I\;\;\;
 \text{(the identity operator)}.
 \end{align}
 As noted in \cite{johnstone2018projective}, the requirement that $G_n = I$ is
 not a very restrictive assumption. For example, one can always enlarge the
 original problem by one operator, setting $A_n=B_n=0$. 

%% file: exampleIntuition.tex
\section{\coapversion{A New Form of }Projective Splitting}

\label{sec_proj_split}
\label{sec_proj_basic}
\coapversion{\subsection{Projective Splitting Basics \myc{in Detail}}}
\input{extSol}
\input{extSolCoap}

\coapversion{At each iteration $k$, projective splitting algorithms work by using a decomposition
procedure to construct a halfspace $H_k \subset \boldsymbol{\mathcal{H}}
\triangleq \mathcal{H}_0\times
\mathcal{H}_1\times\cdots\times\mathcal{H}_{n-1}$ that is guaranteed to
contain $\calS$.  Each new iterate 
$p^{k+1} \in \boldsymbol{\mathcal{H}}$ is
obtained by projecting the previous iterate 
$p^k = (z^{k},w_1^{k},\ldots,w_{n-1}^{k}) \in \boldsymbol{\mathcal{H}}$ 
onto $H_k$, with
possible over- or under-relaxation.}
\preprintversion{A separator-projector algorithm for finding a point in
$\calS$ (and hence a solution to \eqref{prob2}) will, at each iteration $k$,
find a closed and convex set $H_k$ which separates $\calS$ from the current
point, meaning $\calS$ is entirely in the set (preferably, the current point
is not). One can then attempt to ``move closer" to the solution set by
projecting the current point onto the set $H_k$. This general setup guarantees
that the sequence generated by the method is \emph{Fej\'{e}r monotone}
\cite{combettes2000fejerXX} with respect to $\calS$. This alone is not
sufficient to guarantee that the iterates actually converge to a point in the
solution set. To establish this, one needs to show that the set $H_k$
``sufficiently separates" the current point from the solution set, or at least
does so sufficiently often.  Such ``sufficient separation'' allows one to
establish that any weakly convergent subsequence of the iterates must have its
limit in the set $\calS$, from which overall weak convergence follows from
\cite[Prop.~2]{combettes2000fejerXX}.
}

With $\calS$ as in~\eqref{defCompExtSol}, the separator formulation 
presented in~\cite{combettes2016async} constructs 
the halfspace $H_k$ using the function $\varphi_k:\boldsymbol{\mathcal{H}}\to\mathbb{R}$ defined as 
\begin{align} 
\coapAmp\varphi_k(z,w_1,\ldots,w_{n-1}) 
\aCoapLine 
\label{hplane}
&\coapversion{\qquad\qquad}\triangleq \sum_{i=1}^{n-1}
                     \inner{G_i z-x^k_i}{y^k_i-w_i} 
                    +\Inner{z-x_i^n}{y_i^n + \sum_{i=1}^{n-1} G_i^* w_i}\\
\label{hplane_alt}
&\coapversion{\qquad\qquad}= \Inner{z}{\sum_{i=1}^n G_i^* y_i^k} 
+ \sum_{i=1}^{n-1}\inner{x_i^k - G_i x_n^k}{w_i}
          - \sum_{i=1}^n \inner{x_i^k}{y_i^k},
\end{align} 
for some \myc{auxiliary points} ($x_i^k,y_i^k) \in \calH_i^2$. These points ($x_i^k,y_i^k$) will be specified later and must be
\myc{chosen at each iteration} in a specific 
\myc{manner guaranteeing the validity of the separator and} convergence to $\calS$. 
\myc{Among other properties}, they must be chosen so 
that $y^k_i\in T_i x^k_i$ for $i = 1,\ldots,n$. \myc{Under this condition,
it follows readily that $\varphi_k$ has the promised separator properties:}

\myc{
\begin{lemma} \label{lem:sepvalid}
The function $\varphi_k$ defined in~\eqref{hplane} is affine, and if $y^k_i\in
T_i x^k_i$ for all $i = 1,\ldots,n$, then $\varphi_k(z,w_1,\ldots,w_{n-1})
\leq 0$ for all $(z,w_1,\ldots,w_{n-1}) \in
\calS$.
\end{lemma}
\begin{proof}
That $\varphi_k$ is affine is clear from its expression
in~\eqref{hplane_alt}.  Now suppose that $y^k_i\in
T_i x^k_i$ for all $i = 1,\ldots,n$ and $p = (z,w_1,\ldots,w_{n-1}) \in
\calS$.  Then
\begin{align} \label{provesep}
\varphi_k(p)
&= 
- \left(\sum_{i=1}^{n-1} \inner{G_i z-x^k_i}{w_i-y^k_i} 
      + \Inner{z-x_i^n}{w_n - y_i^n}\right),
\end{align}
where
$
w_n \triangleq -\sum_{i=1}^{n-1} G_i^* w_i.
$
From $(z,w_1,\ldots,w_{n-1}) \in \calS$ and the definition of $\calS$, one has
that $w_i \in T_i z$ for all $i=1,\ldots,n-1$, as well as $w_n \in T_n z$.
Since $y_i \in T_i x_i$ for $i=1,\ldots,n$, it follows from the monotonicity
of $T_1,\ldots,T_n$ that every inner product displayed in~\eqref{provesep} is
nonnegative, and so $\varphi_k(p) \leq 0$.  \ourqed
\end{proof}

}


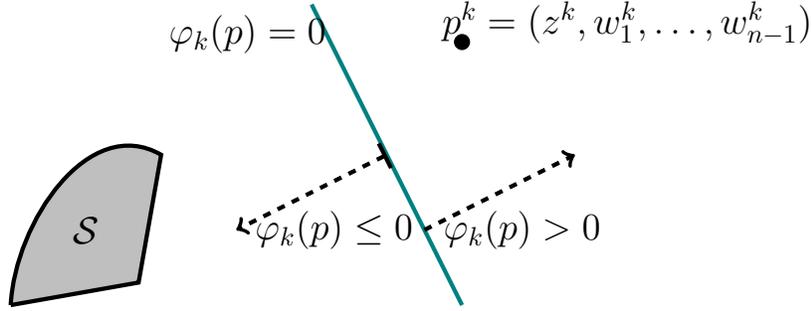
\begin{figure}
	\centering 
	{\large
	\begin{tikzpicture}
	\draw [ultra thick, fill=lightgray] (0,0) to [out=87,in=150] (2,2) 
	    -- (1.70,.30) -- (0,0);
	\node at (1,1) {$\calS$} ;
	\draw [ultra thick, teal] (4,4) -- (6,0) ;
	\draw [fill] (6,3.5) circle [radius=0.1] ;
	\node at (8.2,3.75) {$p^k = (z^{k},w_1^{k},\ldots,w_{n-1}^{k})$} ;
	\node at (3.15,3.6) {$\varphi_k(p) = 0$} ;
	\draw [ultra thick, dashed, ->] (5.5,1) -- (7.5,2) ;
	\node at (6.8,1) {$\varphi_k(p) > 0$} ;
	\draw [ultra thick, dashed, |->] (5,2) -- (3,1) ;
	\node at (4.3,1) {$\varphi_k(p) \leq 0$} ;
	\end{tikzpicture}
	}
	\caption{\myc{Properties of the hyperplane
	$\set{p\in\boldsymbol{\mathcal{H}}}{\varphi_k(p) = 0}$ obtained from
	the affine function $\varphi_k$. This hyperplane is the boundary of the
	halfspace $H_k$, and it always holds that $\varphi_k(p^*) \leq 0$ for
	every $p^*\in\calS$.  When $\varphi_k(p^k) > 0$ (as shown), the hyperplane
	separates the current point $p^k$ from the solution set $\calS$.}}
	\label{fig:hplane}
\end{figure}
\begin{figure}
    \centering				 
	{\large
	\begin{tikzpicture}
	\draw [ultra thick, teal] (1.6,3.75) -- (5.35,0) ;
	\draw [fill] (3.6,1.75) circle [radius=0.1] ;
	\draw [ultra thick, dashed, olive, ->] (4.6,2.75) -- (3.6,1.75) ;
	\draw [ultra thick, teal] (4,4) -- (6,0) ;
	\draw [fill] (4.6,2.75) circle [radius=0.1] ;
	\draw [ultra thick, dashed, olive, ->] (6,3.5) -- (4.6,2.75) ;
	\draw [fill] (6,3.5) circle [radius=0.1] ;
	\node at (8.2,3.75) {$p^k = (z^{k},w_1^{k},\ldots,w_{n-1}^{k})$} ;
	\node at (3.15,3.6) {$\varphi_k(p) = 0$} ;
	\node at (5.3,2.65) {$p^{k+1}$} ;
	\node at (4.27,1.77) {$p^{k+2}$} ;
	\node at (1.15,2.8) {$\varphi_{k+1}(p) = 0$} ;
	\draw [ultra thick, fill=lightgray] (0,0) to [out=87,in=150] (2,2) 
	    -- (1.70,.30) -- (0,0);
	\node at (1,1) {$\calS$} ;
	\end{tikzpicture}
	}
 	\caption{\myc{The basic operation of the method. Each iteration $k$
 		constructs  a separator $\varphi_k$ as shown in
 		Figure~\ref{fig:hplane} and then obtains the next iteration by
 		projecting onto the halfspace $H_k =
 		\set{p\in\boldsymbol{\mathcal{H}}}{\varphi_k(p) \leq 0}$, within which
 		the solution set $\calS$ is known to lie.
    }}
    \label{fig:operation}
\end{figure}
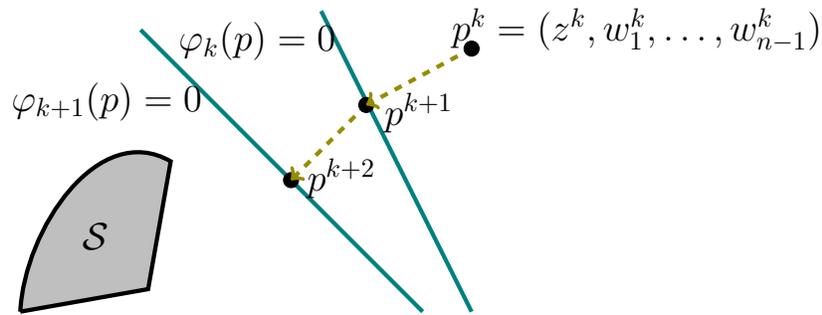 

\myc{
Figure~\ref{fig:hplane} presents a rough depiction of the current algorithm
iterate $p^k = (z^{k},w_1^{k},\ldots,w_{n-1}^{k})$ and the separator
$\varphi_k$ in the case that $\varphi_k(p^k) > 0$. The basic iterative cycle
pursued by projective splitting methods is:
\begin{enumerate}
\item \label{step:process} For each operator $T_i$, identify a pair $(x_i^k,y_i^k)\in\gra T_i$.  These pairs define an affine function $\varphi_k$ such that
$\varphi_k(p) \leq 0$ for all $p\in\calS$, using the
construction~\eqref{hplane} (or related constructions for variations of the
basic problem formulation).
\item \label{step:project}  Obtain the next iterate $p^{k+1}$ by projecting
the current iterate $p^k$ onto the halfspace $H_k \triangleq
\set{p}{\varphi_k(p) \leq 0}$, with possible over- or under-relaxation.
\end{enumerate}

Figure~\ref{fig:operation} presents a rough depiction of two iterations of
this process in the absence of over- or under-relaxation.  The projection
operation in part~\ref{step:project} of th{}e cycle is a straightforward
application of standard formulas for projecting onto a halfspace.  For the
particular formulation~\eqref{prob2}, the necessary calculations are derived
in~\cite{johnstone2018projective} and displayed in
Algorithm~\ref{AlgProjUpdate} below.  This projection is}
a low-complexity
operation involving only inner products, norms, matrix multiplication by $G_i$,
and sums of scalars. For example, when $\calH_i=\rR^d$ for $i=1,\ldots,n$ and each
$G_i=I$, then the projection step has computational complexity $\bigO(nd)$.

\myc{The key question in the design of algorithms 
in this class therefore concerns step~\ref{step:process} in the cycle: 
how might one select the points $(x_i^k,y_i^k)\in\gra
T_i$ so that convergence to $\calS$ may be established? 
The usual approach has been to choose
$(x^k_i,y^k_i)\in\gra T_i$ to be some function of $(z^k,w_i^k)$ such that
$\varphi_k(p^k)$ is positive and ``sufficiently large" whenever $p^k
\not\in \calS$. Then
projecting the current point
onto this hyperplane makes progress toward the solution and can be shown to
lead (with some further analysis) to overall convergence.}  
In the original versions of projective
splitting, the calculation of $(x^k_i,y^k_i)$ involved (perhaps approximately)
evaluating a resolvent; later \cite{johnstone2018projective} introduced the
alternative of a two-forward-step calculation for Lipschitz continuous
operators \myc{that achieved essentially the same sufficient separation condition}.  %



Here, we introduce a one-forward-step calculation for the case
of cocoercive operators. A principal difference between this analysis and earlier work on
projective splitting is that processing all the operators $T_1,\ldots,T_n$ at iteration $k$ need not result in $\varphi_k(p^k)$ being positive.
Instead, we establish an ``ascent lemma" that relates the values
$\varphi_k(p^{k})$ and $\varphi_{k-1}(p^{k-1})$ in such a way that overall
convergence may still be proved, even though it is possible that
$\varphi_k(p^k) \leq 0$ \myc{at some iterations $k$}.  
\myc{In particular,  $\varphi_k(p^k)$ will be larger than the previous value $\varphi_{k-1}(p^{k-1})$, up to some error term that vanishes as $k\to\infty$}.

When $\varphi_k(p^k) \leq 0$, projection onto $H_k =
\set{p}{\varphi_k(p) \leq 0}$ results in
$p^{k+1} = p^k$.  In this case, the algorithm continues to compute
new points $(x_i^{k+1},y_i^{k+1})$, $(x_i^{k+2},y_i^{k+2}),\ldots$ until, for
some $\ell \geq 0$, it constructs a hyperplane $H_{k+\ell}$ such that the
$\varphi_{k+\ell}(p^k) > 0$ and projection results in $p^{k+\ell+1} \neq
p^{k+\ell} = p^k$.

\subsubsection*{Additional Notation for Projective Splitting}
For an arbitrary $(w_1,w_2,\ldots,w_{n-1})
\in\calH_1\times\calH_2\times\ldots\times\calH_{n-1}$ we use the notation
\begin{align*}
w_{n} \triangleq - \sum_{i=1}^{n-1} G_i^*w_i,
\end{align*}
\myc{as in the proof of Lemma~\ref{lem:sepvalid}.}
Note that when $n=1$, $w_1=0$. 
Under \myc{the above} convention, 
we may write $\varphi_k:\bcalH\to\rR$ in the more compact
form
\begin{align*} 
\varphi_k(z,w_1,\ldots,w_{n-1}) &= \sum_{i=1}^{n}\inner{G_i z-x^k_i}{y^k_i-w_i} .
\end{align*} 
We also use the following notation for $i=1,\ldots,n$:
\begin{align*}
\varphi_{i,k}(z,w_i) \triangleq  \langle G_i z-x_i^k,y_i^k - w_i\rangle.
\end{align*}
Note that 
$\varphi_k(z,w_1,\ldots,w_{n-1})=\sum_{i=1}^n \varphi_{i,k}(z,w_i)$.

%% file: extSol.tex
The goal of our algorithm will be to find a point in
\begin{equation} \label{defCompExtSol}
\calS \triangleq \set{(z,w_1,\ldots,w_{n-1})\in\bcalH}{(\forall\,i\in\{ 1,\dots,n-1\})\; w_i\in T_i G_i z, 	
	\; \textstyle{-\sum_{i=1}^{n-1} G_i^* w_i \in T_n z}}.
\end{equation}
It is clear that $z^*$ solves \eqref{prob1}--\eqref{prob2} if and only if there
exist $w_1^*,\ldots,w_{n-1}^*$ such that 
$$
(z^*,w_1^*,\ldots,w_{n-1}^*)\in \calS. $$ Under reasonable assumptions, the
set $\calS$ is closed and convex; see Lemma \ref{lemClosed}. $\calS$ is often called the \emph{Kuhn-Tucker solution set} of problem \eqref{prob2}.

%% file: extSolCoap.tex
{}

%% file: introProposedUpdate.tex
\preprintversion{
\subsection{The New Procedure}\label{sec_genesis}
Suppose $A_i=0$ for some $i\in\{1,\ldots,n\}$. Since $B_i$ is cocoercive, it is also Lipschitz continuous. In \cite{johnstone2018projective} we introduced the following two-forward-step update for Lipschitz continuous $B_i$:
\begin{align*}
x_i^k &= G_i z^k - \rho_i^k(B_i G_i z^k - w_i^k)
\\
y_i^k &= B_i x_i^k.
\end{align*}
Under $L_i$-Lipschitz continuity and the condition $\rho_i^k<1/L_i$, it is
possible to show that updating $(x_i^k,y_i^k)$ in this way leads to
$\varphi_{i,k}(z^k,w_i^k)$ being sufficiently positive to establish overall
convergence. Although we did not discuss it in \cite{johnstone2018projective},
this two-forward step procedure can be extended to handle nonzero $A_i$ in the
following manner:
\begin{align}\label{eqOldForward1prox}
x_i^k +\rho_i^k a_i^k
&=
G_i z^k - \rho_i^k(B_i G_i z^k - w_i^k):\quad a_i^k\in A_i x_i^k
\\\label{eqOldForward1prox2}
y_i^k &= a_i^k + B_i x_i^k.
\end{align}
Following \eqref{defprox2}, it is clear that~\eqref{eqOldForward1prox} is
essentially a resolvent calculation applied to its right-hand side $G_i z^k -
\rho_i^k(B_i G_i z^k - w_i^k)$. This type of update, with forward steps
and backward steps together, was introduced in \cite{tran2015new} for a more
limited form of projective splitting.

An obvious drawback of \eqref{eqOldForward1prox}--\eqref{eqOldForward1prox2}
is that it requires two forward steps per iteration, one to compute $B_i G_i
z^k$ and another to compute $B_i x_i^k$. The initial motivation for the current
paper was the following question: is there a way to reuse $B_i x_i^{k-1}$ so
as to avoid computing $B_i G_i z^k$ at each iteration, perhaps under the
stronger assumption of cocoercivity? With some effort we arrived at the
following update for each block $i=1,\ldots,n$ at each iteration $k\geq 0$:}
\coapversion{
\vspace{0.1cm}

\subsection{The New Procedure}
\myc{ 
Recall the original problem of interest \eqref{prob1}, which is related to \eqref{prob2} via $T_i = A_i + B_i$. 
}
At each iteration $k$ and for each block $i=1,\ldots,n$, we
propose to find a pair $(x_i^k,y_i^k)\in \gra T_i = \gra (A_i+B_i)$ conforming to the conditions}
\begin{align}\label{eqNewUpdate1}
x_i^k +\rho_i a_i^k
&=
(1-\alpha_i)x_i^{k-1}
+
\alpha_i G_i z^k
-
\rho_i
\left(
b_i^{k-1}
- w_i^k
\right):
\quad
a_i^k\in A_i x_i^k
\\\label{eqNewUpdate2a}
b_i^k &= B_i x_i^k
\\\label{eqNewUpdate2}
y_i^k &= a_i^k+b_i^k,
\end{align}
where $\alpha_i\in(0,1)$, $\rho_i\leq 2(1-\alpha_i)/L_i$, and $b_i^0 =
B_i x_i^0$.  Condition~\eqref{eqNewUpdate1} is readily satisfied by some
simple linear algebra calculations and a resolvent calculation
involving $A_i$.  
\myc{
In particular, referring to~\eqref{defprox2}, 
one may see that \eqref{eqNewUpdate1} is equivalent to computing
\begin{align*}
t &= 
(1-\alpha_i)x_i^{k-1}
+
\alpha_i G_i z^k
-
\rho_i
\left(
b_i^{k-1}
- w_i^k
\right)
\\
x_i^k 
&=
J_{\rho_i A_i}(
t
)
\\
a_i^k &= (1/\rho_i)\left(
t-x_i^k
\right).
\end{align*}
Following this resolvent calculation}, \eqref{eqNewUpdate2a} requires only an
evaluation (forward step) on $B_i$, and \eqref{eqNewUpdate2} is a simple
vector addition.
\preprintversion{In
comparison to~\eqref{eqOldForward1prox}, we have replaced $B_i G_i z^k$ with
the previously computed point $B_i x_i^{k-1}$. However, in order to establish
convergence, it turns out that we also need to replace $G_i z^k$ with a convex
combination of $x_i^{k-1}$ and $G_i z^k$.} 

\myc{The parameter $\rho_i$ plays the role of the stepsize in the resolvent
calculation.  It also plays the role of a forward (gradient) stepsize, since
it multiplies $-b_i^{k-1}$ in~\eqref{eqNewUpdate1}, and $b_i^{k-1}=B_i
x_i^{k-1}$ by~\eqref{eqNewUpdate2a}.  From the assumptions on $\alpha_i$ and
$\rho_i$ immediately following~\ref{eqNewUpdate2}, it follows that $\rho_i$
may be} made arbitrarily close to $2/L_i$ by setting $\alpha_i$ close to $0$.
However, in practice it may be better to use an intermediate value, such as
$\alpha_i=0.1$, since \myc{doing so causes} 
the update to make significant use of the
information in $z^k$, a point computed more recently than $x_i^{k-1}$.

Computing $(x_i^k,y_i^k)$ 
\myc{as proposed in~\eqref{eqNewUpdate1}-\eqref{eqNewUpdate2}
does not guarantee that the quantity}
$\varphi_{i,k}(z^k,w_i^k)$ is positive. 
\preprintversion{In the next}\coapversion{In the remainder of this} section, we give
some intuition as to why \eqref{eqNewUpdate1}-\eqref{eqNewUpdate2}
nevertheless leads to convergence to $\calS$.

%% file: onePGstep.tex
\preprintversion{\subsection{A Connection with the Forward-Backward Method}}
\label{secConnect2Prox}

In the projective splitting literature preceeding
\cite{johnstone2018projective}, the pairs $(x_i^k, y_i^k)$ are solutions of
\begin{align}\label{eqOldProx}
x_i^k+\rho_i y_i^k = G_i z^k + \rho_i w_i^k:\quad y_i^k\in T_i x_i^k
\end{align}
for some $\rho_i>0$, which --- \myc{again} following \eqref{defprox2} --- is a resolvent
calculation. 
\preprintversion{It can be shown that the resulting $(x_i^k,y_i^k)\in\gra T_i$ are
such that $\varphi_{i,k}(z^k,w_i^k)$ is positive and sufficiently large to
guarantee overall convergence to a solution of \eqref{prob2}.}
Since the
stepsize $\rho_i$ in~\eqref{eqOldProx}
can be any positive number, let us replace $\rho_i$ with
$\rho_i/\alpha_i$ for some $\alpha_i\in(0,1)$ and rewrite
\eqref{eqOldProx} as
\begin{align}\label{eqOldProx2}
x_i^k+\frac{\rho_i}{\alpha_i} y_i^k = G_i z^k + \frac{\rho_i}{\alpha_i} w_i^k:\quad y_i^k\in T_i x_i^k.
\end{align}
The reason for this reparameterization will become apparent below. 

In this paper, $T_i = A_i+B_i$, with $B_i$ being cocoercive and $A_i$ maximal
monotone. For $T_i$ in this form, computing the resolvent as in
\eqref{eqOldProx} exactly may be impossible, even when the resolvent of $A_i$
is available. With this structure, $x_i^k$ in \eqref{eqOldProx2} satisfies:
\begin{align*} 
0 &
=\myc{
\frac{\rho_i}{\alpha_i} y_i^k + x_i^k 
-
\left(G_i z^k + \frac{\rho_i}{\alpha_i} w_i^k\right)}
\\
\implies
0&\in \frac{\rho_i}{\alpha_i} A_i x_i^k+ \frac{\rho_i}{\alpha_i} B_i x_i^k+ x_i^k - \left(G_i z^k + \frac{\rho_i}{\alpha_i} w_i^k\right)
\end{align*}
which can be rearranged to
$
0\in A_i x_i^k +\tilde{B}_i x_i^k,
$
where 
$$
\tilde{B}_i v = B_i v + \frac{\alpha_i}{\rho_i}\left(v - G_i z^k - \frac{\rho_i}{\alpha_i} w_i^k\right). 
$$ 
Since $B_i$ is $L_i^{-1}$-cocoercive, $\tilde{B}_i$ is
$(L_i+\alpha_i/\rho_i)^{-1}$-cocoercive
\cite[Prop.~4.12]{bauschke2011convex}. Consider the generic monotone inclusion
problem $0\in A_i x +\tilde{B}_i x$:
$A_i$ is maximal and $\tilde{B}_i$ is cocoercive, and thus one may solve the
problem with the forward-backward (FB) method~\cite[Theorem
26.14]{bauschke2011convex}. If one applies a single iteration of FB
initialized at $x_i^{k-1}$, with stepsize $\rho_i$, to the inclusion $0\in
A_i x +\tilde{B}_i x$, one obtains the calculation:
\begin{align*}
x_i^k &= 
J_{\rho_i A_i}
\left(
x_i^{k-1}
-
\rho_i
\tilde{B}_i x_i^{k-1}
\right)
\\
&=
J_{\rho_i A_i}
\left(
x_i^{k-1}
-
\rho_i
\left(
B_i x_i^{k-1}+
\frac{\alpha_i}{\rho_i}
\left(
x_i^{k-1} - G_i z^k - \frac{\rho_i}{\alpha_i} w_i^k
\right)
\right)
\right)
\\
&=
J_{\rho_i A_i}
\left(
(1-\alpha_i)x_i^{k-1}
+
\alpha_i G_i z^k
-
\rho_i
(
B_i x_i^{k-1}
- w_i^k
)
\right),
\end{align*}
which is precisely the update \eqref{eqNewUpdate1}. So, our proposed
calculation is equivalent to \emph{one} iteration of FB initialized at the previous
point $x_i^{k-1}$, applied to the subproblem of computing the resolvent in
\eqref{eqOldProx2}. Prior versions of projective splitting require computing
this resolvent either exactly or to within a certain relative error criterion,
which may be \myc{time consuming}.  
Here, we simply make a single FB step toward computing
the resolvent, which we will prove is sufficient for the projective splitting
method to converge to $\calS$. 
However, our stepsize restriction on $\rho_i$
will be slightly stronger than the natural stepsize limit that would arise when
applying FB to $0\in A_i x +\tilde{B}_i x$.



%% file: mainAssump.tex
\section{The Algorithm}
\label{sec_algAlg}
\label{secMainAss}

\subsection{Main Problem Assumptions and Preliminary Results}
\begin{assumption}
	\label{AssMonoProb}\label{assMono}	
	Problem~(\ref{prob1}) conforms to the following:
	\begin{enumerate}
	\item $\mathcal{H}_0 = \mathcal{H}_n$ and
	$\mathcal{H}_1,\ldots,\mathcal{H}_{n-1}$ are real Hilbert spaces.
	\item For $i=1,\ldots,n$, the operators
	$A_i:\mathcal{H}_{i}\to2^{\mathcal{H}_{i}}$ and
	$B_i:\mathcal{H}_i\to\mathcal{H}_i$ are monotone. Additionally each $A_i$
	is maximal.
	\item Each operator $B_i$ is either
	$L_i^{-1}$-cocoercive for some $L_i>0$
	\preprintversion{(and thus single-valued)}
	and $\dom B_i = \mathcal{H}_i$, or $L_i = 0$ and $B_ix = v_i$ for all
	$x\in\calH_i$ and some $v_i\in\calH_i$ 
	\myc{(that is, $B_i$ is a constant function)}.
	\item Each $G_i:\calH_{0}\to\calH_i$ for $i=1,\ldots,n-1$ is
	linear and bounded. 
	\item Problem~\eqref{prob1} has a solution, so the set $\calS$
	defined in
	\eqref{defCompExtSol} is nonempty.
	\end{enumerate}
\end{assumption}
\myc{ 
Problem \eqref{ProbOpt} will be equivalent to an instance of Problem \eqref{prob1} satisfying Assumption \ref{AssMonoProb} if each $f_i$ and $h_i$ is closed, convex, and proper, each $h_i$ has $L_i$-Lipschitz continuous gradients, and a special case of the constraint qualification in \cite[Prop.~5.3]{combettes2013systems} holds.
}

\preprintversion{In order to apply a separator-projector algorithm, the target
set must be closed and convex. Establishing this for $\calS$ is very similar
to in our previous work \cite{johnstone2018projective}, which in turn follows
many earlier results.}

\begin{lemma}
\label{lemClosed}
Suppose Assumption \ref{AssMonoProb} holds. The set $\calS$ defined in \eqref{defCompExtSol} is closed and convex. 
\end{lemma} 
\begin{proof}
	By \cite[Cor.~20.28]{bauschke2011convex} each $B_i$ is maximal.
Furthermore, since $\dom(B_i)=\calH_i$, $T_i=A_i+B_i$ is maximal monotone by
\cite[Cor.~25.5(i)]{bauschke2011convex}. The rest of the proof is identical to
\cite[Lemma 3]{johnstone2018projective}.
\ourqed
\end{proof}

%% file: theHyperplane2.tex
Throughout, we will use $p =(z,\bw) =  (z,w_1,\ldots, w_{n-1})$ for a
generic point in $\bcH$, the collective primal-dual space. For $\bcH$, we
adopt the following (standard) norm and inner product:
\begin{align} \label{gammanorm}
\norm{(z,\bw)}^2 &\triangleq \|z\|^2 + \sum_{i=1}^{n-1}\|w_i\|^2 &
\Inner{(z^1,\bw^1)}{(z^2,\bw^2)} &\triangleq 
\langle z^1,z^2\rangle + \sum_{i=1}^{n-1}\langle
w^1_i,w^2_i\rangle.
\end{align}

\begin{lemma}
\label{LemGradAffine} 
\emph{\cite[Lemma 4]{johnstone2018projective}} Let $\varphi_k$ be defined as in (\ref{hplane}).  Then:
\begin{enumerate}
	\item $\varphi_k$ is affine on $\bcH$. 
	\item \label{item:gradForm} With respect to inner product
	$\langle\cdot,\cdot\rangle$ on $\bcH$, the gradient
	of $\varphi_k$ is
	$$
	\nabla\varphi_k =
	\left(\sum_{i=1}^{n-1} G_i^* y_i^k+ y_n^k ,
	x_1^k - G_1 x_{n}^k,x_2^k - G_2 x_{n}^k,\ldots,x_{n-1}^k - G_{n-1} x_{n}^k\right).
	$$
\end{enumerate}
\end{lemma}

%% file: genericUp.tex
\subsection{Abstract One-Forward-Step Update}


We sharpen the notation for the one-forward-step update introduced in
\eqref{eqNewUpdate1}--\eqref{eqNewUpdate2} as follows:
\begin{definition}\label{defUp}
	Suppose $\calH$ and $\calH'$ are real Hilbert spaces, $A:\calH\to2^\calH$ is maximal-monotone with nonempty domain, $B:\calH\to\calH$ is $L^{-1}$-cocoercive, and $G:\calH'\to\calH$ is bounded and linear. For $\alpha\in[0,1]$ and $\rho>0$, define the  mapping $\calF_{\alpha,\rho}(z,x,w;A,B,G):\calH'\times\calH^2\to \calH^2$, with additional parameters $A,B$, and $G$, as
	\begin{align}\label{eqF}	
	\calF_{\alpha,\rho}\!\left(
	\begin{array}{c} 
	z,x,w; \\ A,B,G
	\end{array}
	\right) :
	&= (x^+,y^+) :
	\left\{
	\begin{array}{ll}
	t &\triangleq(1-\alpha)x+\alpha G z -\rho(B x - w)
	\\
	x^+ &= J_{\rho A}\left(t\right)
	\\
	y^+ &= \rho^{-1}(t - x^+) +B x^+.
	\end{array}
	\right. 
	\end{align}
\end{definition}
\myc{To simplify the presentation, we will also use the notation
\begin{align}\label{eqFsimp}
\calF^i(z,x,w)\triangleq 
\calF_{\alpha_i,\rho_i}\left(
z,x,w; A_i,B_i,G_i
\right).
\end{align}
\noindent With this notation, \eqref{eqNewUpdate1}--\eqref{eqNewUpdate2} may be written as $
(x_i^k,y_i^k) = \calF^i(z^k,x_i^{k-1},w_i^k).
$
}

%% file: allAlgDefs.tex
\subsection{Algorithm Definition}
\label{sec_alg}

\begin{algorithm}[ht!]	\label{AlgfullWithBT}
	\DontPrintSemicolon
	\SetKwInOut{Input}{Input}
	\SetKwFunction{backTrack}{backTrack}
	\SetKwFunction{projectToHplane}{projectToHplane}
	\caption{One-Forward-Step Projective Splitting with Backtracking}
	\Input{$(z^1,{\bf w}^1)\in \boldsymbol{\mathcal{H}}$, $\mathcal{B}\subseteq\{1,\ldots,n\}$ (the operators requiring backtracking), $\gamma>0$, $\delta\in(0,1)$, and $\hat{\rho}$. For $i=1,\ldots,n$:  $x_i^0\in\calH_i$ 
	\myc{and $0< \alpha_i\leq 1$}.
For $i\in\mathcal{B}$: \myc{$\rho_i^0>0$}
		$\hat{\theta}_i\in\dom(A_i)$,  $\hat{w}_i\in A_i\hat{\theta}_i+B_i\hat{\theta}_i$, 
		and $y_i^0\in A_i x_i^0+B_i x_i^0$.
		For $i\notin\mathcal{B}$: $\rho_i>0$.
		}	
		For $i\in\mathcal{B}$ set $\eta_i^0=0$\;
	\For{$k=1,2,\ldots$}
	{   
		\For{$i\in\mathcal{B}$}
		{								
			$\label{lineBeginBack}(x_i^k,y_i^k, \rho_i^{k}, \eta_i^{k})=$~\backTrack{$z^k$, $x_i^{k-1}$, $w_i^k$, $y_i^{k-1}$, $\rho_i^{k-1}$, $\eta_i^{k-1}$; $i$}\label{lineDoBT}\; 
			\tcc{\backTrack defined in Algorithm \ref{AlgBackTrack}}
		}
	
		\For{$i\notin\mathcal{B}$}
		{
		\myc{$(x_i^k,y_i^k) = \calF^i(z^k, x_i^{k-1}, w_i^k)\qquad$ 
		\tcc{$\calF^i$ defined in \eqref{eqF}-\eqref{eqFsimp}}}\label{lineDoNoBT}
		}
				
		$(\pi_k, z^{k+1}, {\bf w}^{k+1})=$ \projectToHplane{$z^k$, ${\bf w}^k$, $\{x_i^k,y_i^k\}_{i=1}^n$} \label{lineproject}\;
		\tcc{\projectToHplane defined in Algorithm \ref{AlgProjUpdate}}		
		\If{$\pi_k=0$}
		{
			\Return $z^{k+1}$\label{lineReturn}\;
			
		} 
	}	
\end{algorithm}

\begin{algorithm}[ht!]\label{AlgBackTrack}
	\DontPrintSemicolon
	\SetKwInOut{Input}{Input}
	\SetKwInOut{GlobalVars}{Global Variables for Function}
	\SetKwFunction{backTrack}{backTrack}
	\caption{Backtracking procedure}
	\label{AlgLineSearch}
		\GlobalVars{$G_i, A_i, B_i$,
		$\alpha_i$, $\hat{\theta}_i$, and $\hat{w}_i$ 
		for $i\in \mathcal{B}$, 
		$\delta$ and $\hat{\rho}$.
	    }
	\SetKwProg{Fn}{Function}{:}{}
	\Fn{\backTrack{$z$, $x$, $w$, $y$, $\rho$, $\eta$; $i$}}{
		$A=A_i$, $B=B_i$, $G=G_i$, $\alpha = \alpha_i$, $\hat{\theta} = \hat{\theta}_i$, $\hat{w} = \hat{w}_i$\;
		$\varphi =  \langle G z-x ,y - w\rangle$\label{lineGetPhi}\;
		$\orho = \min\{(1+\alpha\eta)\rho,\hat{\rho}\}$\;
		Choose $\tilde{\rho}_{1} \in
		 [\rho,\orho]$\label{lineChooseRho}\;
		\For{$j=1,2,\ldots$}{\label{lineBeginFor}
			$(\tilde{x}_{j},\tilde{y}_{j}) = \calF_{\alpha,\tilde{\rho}_{j}}(z,x,w;A,B,G)\qquad$ \tcc{$\calF$ defined in \eqref{eqF}}\label{lineDoF}
			$\hat{y}_{j} = \tilde{\rho}_j^{-1}\left((1-\alpha)x+\alpha Gz-\tilde{x}_j\right)+w$\;
			$\varphi^+_j = \langle  G z- \tilde{x}_{j},\tilde{y}_{j} - w\rangle 
			$\;
			\If{\label{lineC1}$\|\tilde{x}_{j} - \hat{\theta}\|\leq 	(1-\alpha)\|x-\hat{\theta}\|
				+
				\alpha\| G z- \hat{\theta}\|+\tilde{\rho}_{j}\|w - \hat{w}\|$
			 \\ ~~~~\emph{and} $\varphi_j^+
			 \geq			 
			 \frac{\tilde{\rho}_{j}}{2\alpha}
			 \left(
			 \|\tilde{y}_{j} - w\|^2
			 +
			 \alpha 
			 \|\hat{y}_{j} - w\|^2
			 \right)+ (1-\alpha)
			 \left(
			 \varphi
			 -
			 \frac{\tilde{\rho}_{j}}{2\alpha}
			 \|y - w\|^2
			 \right)$\label{lineC2}
			}{
			    $\eta = \|\hat{y}_j - w\|^2/\|\tilde{y}_j-w\|^2$\;
				\Return{ $(\tilde{x}_{j},\tilde{y}_j,\tilde{\rho}_j,\eta)$}				
				\label{lineBTreturn}
			}
			$\tilde{\rho}_{j+1} = \delta\tilde{\rho}_{j}$\label{lineStepDec}\;
		}
	}
\end{algorithm}

\begin{algorithm}[ht!]
	\DontPrintSemicolon
	\SetKwInOut{GlobalVars}{Global Variables for Function}
	\SetKwInOut{Input}{Input}
	\SetKwFunction{projectToHplane}{projectToHplane}
	\GlobalVars{$G_i$ for $i=1,\ldots,n-1$, and $\gamma$.}
	\SetKwProg{Fn}{Function}{:}{}
	\Fn{\projectToHplane{$z,\bw$, $\{x_i,y_i\}_{i=1}^n$}}
	{				
		$u_i = x_i - G_i x_n,\quad i=1,\ldots,n-1,$\;
		$v = \sum_{i=1}^{n-1} G_i^* y_i+y_n$\;    
		$\pi = \|u\|^2+\gamma^{-1}\|v\|^2$\label{linePiUpdate}\;    
		\eIf{$\pi>0$}{
			$\varphi(p) = 
			\langle z, v\rangle 
			+
			\sum_{i=1}^{n-1}
			\langle w_i,u_{i}\rangle 
			-
			\sum_{i=1}^{n}
			\langle x_i,y_i\rangle  
			$\;			
			$\tau = \frac{1}{\pi}\cdot\max\left\{0,\varphi(p)\right\}
			$\;
		}
		{	
			\Return $(0, x_n, y_1,\ldots,y_{n-1})$\;
			
		} 
		$z^{+} = z - \gamma^{-1}\tau v$\label{lineZproj}\;
		$w_i^{+} = w_i - \tau u_{i},\quad i=1,\ldots,n-1$,\label{lineWproj}\;
		$w_{n}^{+} = -\sum_{i=1}^{n-1} G_i^* w_{i}^{+}$\;
		\Return $(\pi, z^+, \bw^{+})$
    }
	\caption{Projection Update}
	\label{AlgProjUpdate}
\end{algorithm}

%% file: theAlgorithm.tex
Algorithms \ref{AlgfullWithBT}--\ref{AlgProjUpdate} define the main method
proposed in this work. They produce a sequence of primal-dual iterates $p^k =
(z^k,w_1^k,\ldots,w_{n-1}^k)\in\bcalH$ and, implicitly, $w_n^k \triangleq
-\sum_{i=1}^{n-1} G_i^* w_i^k$. Algorithm~\ref{AlgfullWithBT} gives the basic
outline of our method; for each operator, it invokes either our new
one-forward-step update with a user-defined stepsize (through line~\ref{lineDoNoBT})
or its backtracking variant given in Algorithm~\ref{AlgBackTrack} (through
line~\ref{lineDoBT}).
Together, algorithms \ref{AlgfullWithBT}--\ref{AlgBackTrack} specify
how to update the points $(x_i^k,y_i^k)$ used to define the separating affine
function $\varphi_k$ in
\eqref{hplane}.
Algorithm~\ref{AlgProjUpdate}, called from line~\ref{lineproject} of
Algorithm~\ref{AlgfullWithBT}, defines the \texttt{projectToHplane} function
that performs the projection step to obtain the next iterate.

Taken together, algorithms \ref{AlgfullWithBT}--\ref{AlgProjUpdate} are essentially
the same as Algorithm 2 of \cite{johnstone2018projective}, except that the
update of $(x_i^k,y_i^k)$ uses the new procedure given in
\eqref{eqNewUpdate1}--\eqref{eqNewUpdate2}. For simplicity, the algorithm also
lacks the block-iterative and asynchronous features
of~\cite{combettes2016async,eckstein2017simplified,johnstone2018projective},
which we plan to combine with
algorithms \ref{AlgfullWithBT}--\ref{AlgProjUpdate} in a follow-up paper. 


The
computations in \texttt{projectToHplane} are all straightforward and of
relatively low complexity. They consist of matrix multiplies by $G_i$,
inner products, norms, and sums of scalars. In particular, there are no
potentially difficult minimization problems involved. If
$G_i=I$ and $\calH_i = \mathbb{R}^d$ for $i=1,\ldots,n$, then the
computational complexity of \texttt{projectToHplane} is $\bigO(nd)$.

\subsection{Algorithm Parameters}
The method allows two ways to select the stepsizes $\rho_i$.
One may either choose them manually or invoke the \texttt{backTrack}
procedure. 
If one decides to select the stepsizes
manually, the upper bound condition $\rho_i\leq2(1-\alpha_i)/L_i$ is
required whenever $L_i>0$. However, it may be difficult to ensure that this
condition is satisfied when the cocoercivity constant is hard to estimate. The
global cocoercivity constant $L_i$ may also be conservative in parts of the
domain of $B_i$, leading to unnecessarily small stepsizes in some cases. We
developed the backtracking linesearch technique for these reasons.  The set
$\mathcal{B}$ holds the indices of operators for which backtracking is to be
used.

For a trial stepsize $\tilde{\rho}_{j}$, Algorithm \ref{AlgBackTrack}
generates candidate points $(\tilde{x}_{j},\tilde{y}_{j})$ using the
single-forward-step procedure of \eqref{eqF}. For these candidates, Algorithm
\ref{AlgBackTrack} checks two conditions on lines \ref{lineC1}--\ref{lineC2}.
If both of these inequalities are satisfied, then backtracking terminates and
returns the successful candidate points. If either condition is not satisfied,
the stepsize is reduced by the factor $\delta\in(0,1)$ and the process is
repeated. These two conditions arise in the analysis in Section
\ref{sec_main_proof}.

\myc{ 
The parameter $\hat{\rho}$ is a global upper bound on the stepsizes (both backtracked and fixed) and must be chosen to satisfy Assumption \ref{assStep}.
In \texttt{backTrack}, one must choose an initial trial stepsize within a specified interval
(line \ref{lineChooseRho} of Algorithm \ref{AlgBackTrack}). This interval arises in the analysis (see lemmas \ref{lemAv2step} and \ref{lemAscent1}). 
Written in terms of the parameters passed into \texttt{backTrack} in the call on line \ref{lineDoBT} of Algorithm \ref{AlgfullWithBT}, and assuming the global upper bound $\hat{\rho}$ is sufficiently large to not be active on line \ref{lineChooseRho}, the interval is 
$$
\left[\rho_i^k,\left(1+\alpha_i\frac{\|\hat{y}_i^k - w_i^k\|}{\|y_i^k - w_i^k\|}\right)\rho_i^k\right].
$$
An obvious choice is to set the initial stepsize to be at the upper limit of the interval. 
In practice we have observed that
$
\|y_i^k - w_i^k\|
$ 
and 
$
\|\hat{y}_i^k - w_i^k\|
$
tend to be approximately equal, 
so this allows for an increase in the trial stepsize by up to a factor of 
approximately $1+\alpha_i$ over the previous stepsize.

Note that \texttt{backTrack} returns the chosen stepsize $\tilde{\rho}_j$ as well as the quantity $\eta$ which are needed to compute the available interval in the call to \texttt{backTrack} during the next iteration. 

In the analysis it will be convenient to let $\tilde{\rho}^{(i,k)}$ be the initial trial stepsize chosen  during iteration $k$ of Algorithm \ref{AlgfullWithBT}, when
\texttt{backTrack} has been called through line \ref{lineDoBT} for some
$i\in\mathcal{B}$. 

We call the stepsize returned by \texttt{backTrack} $\rho_i^k$. 
Assuming that \texttt{backTrack} always terminates finitely (which we will
show to be the case), we may write for $i\in\mathcal{B}$
$$
(x_i^k,y_i^k) = \calF_{\alpha_i,\,\rho_i^k}(z^k, x_i^{k-1}, w_i^k; A_i, B_i, G_i)
$$
The only difference between the update for $i\in\mathcal{B}$ on line \ref{lineDoNoBT} and this update for $i\notin\mathcal{B}$ is that in the former, the stepsize $\rho_i^k$ is
discovered by backtracking, while in the latter it is directly user-supplied.
}

The \texttt{backTrack} procedure computes several auxiliary quantities used to
check the two backtracking termination conditions. The point $\hat{y}_{j}$ is
calculated to be the same as $\hat{y}$ given in Definition \ref{defYhat}. The
quantity $\varphi_j^+ = \langle G z - \tilde{x}_{j},\tilde{y}_{j} - w\rangle$
is the value of $\varphi_{i,k}(z^k,w_i^k)$ corresponding to the candidate
points $(\tilde{x}_{j},\tilde{y}_{j})$. The quantity $\varphi$ computed on line \ref{lineGetPhi} is equal to $\varphi_{i,k-1}(z^k,w_i^k) = \langle G_i z^k - x_i^{k-1},y_i^{k-1} - w_i^k\rangle$. Typically, we want $\varphi_j^+$ to be as large as
possible to get a bigger cut with the separating hyperplane, but the condition checked on line \ref{lineC2} will ultimately
suffice to prove convergence.

Algorithm \ref{AlgfullWithBT} has several additional parameters.  
\begin{description}
\item[ $(\hat{\theta}_i,\hat{w}_i)$] these are used in the backtracking procedure for $i\in\mathcal{B}$. An obvious choice which we used in our numerical experiments was 
$(\hat{\theta}_i,\hat{w}_i) = (x_i^0,y_i^0)$, i.e.~the initial point.
\item[$\gamma>0$:] allows for the projection
to be performed using a slightly more general primal-dual metric than
\eqref{gammanorm}. In effect, this parameter changes the relative size of the
primal and dual updates in lines \ref{lineZproj}--\ref{lineWproj} of Algorithm \ref{AlgProjUpdate}. As $\gamma$
increases, a smaller step is taken in the primal and a larger step in the
dual. As $\gamma$ decreases, a smaller step is taken in the dual update and a
larger step is taken in the primal. See \cite[Sec.~5.1]{eckstein2009general}
and \cite[Sec.~4.1]{eckstein2008family} for more details.
\end{description}

\myc{
In Algorithm \ref{AlgfullWithBT}, the averaging parameters $\alpha_i$ and
user-selected stepsizes $\rho_i$ are fixed across all iterations. In the
preprint version of this paper \cite{johnstone2019single}, we instead allow
these parameters to vary by iteration, subject to certain restrictions.  Doing
so complicates the notation and the analysis, so for relative simplicity we
consider only fixed values of these parameter here.  This simplification also
accords with the parameter choices in our computational tests below.  For the
full, more complicated analysis, please refer to~\cite{johnstone2019single}.
}

As written, Algorithm \ref{AlgfullWithBT} is not as efficient as it could be.
On the surface, it seems that we need to recompute $B_i x_i^{k-1}$ in order to
evaluate $\calF$ on line \ref{lineDoNoBT}. However, $B_i x_i^{k-1}$ was
already computed in the previous iteration and can obviously be reused, so
only one evaluation of $B_i$ is needed per iteration. Similarly, within
\texttt{backTrack}, each invocation of $\calF$ on line \ref{lineDoF} may reuse the
quantity $B x = B_i x_i^{k-1}$ which was computed in the previous iteration of
Algorithm \ref{AlgfullWithBT}. Thus, each iteration of the loop within
\texttt{backTrack} requires one new evaluation of $B$, to compute
$B\tilde{x}_j$ within $\calF$.

%% file: stepAssump.tex
\label{secParam}

We now precisely state our stepsize assumption for the manually chosen stepsizes, as well as the stepsize upper bound $\hat{\rho}$.

\begin{assumption}\label{assStep}
	 	For $i\notin \mathcal{B}$: If $L_i>0$, then
		$
		0<\rho_i\leq 2(1-\alpha_i)/L_i,
		$
		otherwise $\rho_i>0$. The parameter $\hat{\rho}$ must satisfy
		\begin{align}\label{hatRhoBound}
		\hat{\rho}\geq \max\left\{\max_{i\in\mathcal{B}} \rho_i^0,\max_{i\notin \mathcal{B}}\rho_i \right\}.
		\end{align}
\end{assumption}
\myc{ 
Note that if $L_i>0$, Assumption \ref{assStep} effectively limits $\alpha_i$
to be strictly less than $1$, otherwise the stepsize $\rho_i$ would be forced
to $0$, which is prohibited. In this case $\alpha_i$ must be chosen in
$(0,1)$. On the other hand, if $L_i=0$, there is no constraint on $\rho_i$
other than that it is positive and nonzero, and in this case $\alpha_i$ may be
chosen in $(0,1]$. }

%% file: sepProperties.tex
\subsection{Separator-Projector Properties}
Lemma \ref{lemIsProject} details the key results for Algorithm
\ref{AlgfullWithBT} that stem from it being a
seperator-projector algorithm. While these properties alone do not guarantee
convergence, they are important to all of the arguments that follow.
\begin{lemma}\label{lemIsProject}
	Suppose that Assumption \ref{assMono} holds.
	Then for Algorithm \ref{AlgfullWithBT}
	\begin{enumerate}
		\item The sequence $\{p^k\}=\{(z^k,w_1^k,\ldots,w_{n-1}^k)\}$ is bounded.
		\item \label{item:successive} If the algorithm never terminates via
		line \ref{lineReturn}, $p^k - p^{k+1}\to 0$. Furthermore $z^k -
		z^{k-1}\to 0$ and $w_i^k - w_i^{k-1}\to 0$ for $i=1,\ldots n$.
		\item If the algorithm never terminates via line \ref{lineReturn} and $\|\nabla\varphi_k\|$ remains bounded for all $k\geq 1$, then $\limsup_{k\to\infty}\varphi_k(p^k)\leq 0$. 
	\end{enumerate}
\end{lemma}
\begin{proof}
Parts 1--2 are proved in lemmas 2 and 6 of \cite{johnstone2018projective}.
Part 3 can be found in Part 1 of the proof of Theorem 1 in
\cite{johnstone2018projective}.  The analysis
in~\cite{johnstone2018projective} uses a different procedure to construct the
pairs $(x_i^k,y_i^k)$, but the result is generic and not dependent on that
particular procedure.  Note also that \cite{johnstone2018projective}
establishes the results in a more general setting allowing asynchrony and
block-iterativeness, which we do not analyze here. \ourqed
\end{proof}

%% file: nIsOne.tex
\section{The Special Case $n=1$}
Before starting the analysis, we consider the important special case $n=1$. In
this case, we have by assumption that $G_1=I$, $w_1^k=0$, and we are solving
the problem
$
0\in A z + B z,
$
where both operators are maximal monotone and $B$ is $L^{-1}$-cocoercive. In
this case, Algorithm \ref{AlgfullWithBT}  reduces to a method which is similar
to FB. Let $x^k\triangleq x_1^k$, $y^k\triangleq y_1^k$, $\alpha\triangleq
\alpha_1$, and $\rho\triangleq \rho_1$. Assuming for simplicity that $\mathcal{B}=\{\emptyset\}$, meaning backtracking is not being used, then the updates carried out by
the algorithm are
\begin{align}
x^k &= J_{\rho A}
\left(
(1-\alpha)x^{k-1}+\alpha z^k -\rho B x^{k-1}
\right)
\label{eqReduce2FB}\\\nonumber 
y^k &= B x^k + \frac{1}{\rho}\left(
(1-\alpha)x^{k-1}+\alpha z^k -\rho B x^{k-1}-x^k
\right)
\\\nonumber 
z^{k+1}&=  z^k - \tau^k y^k,\quad 
\text{where~} \tau^k = \frac{\max\{\langle z^k - x^k,y^k\rangle,0\}}{\|y^k\|^2}.
\end{align}
If $\alpha=0$, then for all $k\geq 2$, the iterates computed in \eqref{eqReduce2FB} reduce simply to
\begin{align*}
x^k = J_{\rho A}\left( x^{k-1} - \rho B x^{k-1}\right)
\end{align*}
which is exactly FB. However, $\alpha=0$ is not allowed in our analysis.
Thus, FB is a forbidden boundary case which may be approached by setting
$\alpha$ arbitrarily close to $0$. As $\alpha$
approaches $0$, the stepsize constraint $\rho\leq 2(1-\alpha)/L$
approaches the classical stepsize constraint for FB: 
$\rho\leq2/L-\epsilon$ for some arbitrarily small constant $\epsilon>0$. A 
potential
benefit of Algorithm \ref{AlgfullWithBT} over FB in the $n=1$ case is that it
does allow for backtracking when $L$ is unknown or only a conservative
estimate is available.


%% file: mainProof.tex
\section{Main Proof}\label{sec_main_proof}

The  core of the proof strategy will be to establish \eqref{eqAlltheMarbles}
below. If this can be done, then weak convergence to a solution follows from
part 3 of Theorem 1 in \cite{johnstone2018projective}.
\begin{lemma}\label{lemStrat}
Suppose Assumption \ref{assMono} holds and Algorithm
\ref{AlgfullWithBT} produces an infinite sequence of iterations without
terminating via Line \ref{lineReturn}. If
\begin{align}\label{eqAlltheMarbles}
(\forall i=1,\ldots,n):\quad y_i^k-w_i^k\to 0
\text{ and }\,\,G_i z^k - x_i^k\to 0,
\end{align}
then there exists $(\overline{z},\overline{\bw})\in\calS$ such that 
$(z^k,\bw^k)\rightharpoonup (\overline{z},\overline{\bw})$. 
Furthermore, 
\myc{we also have} $x_i^k\rightharpoonup G_i \bar z$
and $y_i^k\rightharpoonup \overline{w}_i$ for all $i=1,\ldots,n-1$,
$x_{n}^k\rightharpoonup \bar z$, and $y_n^k\rightharpoonup
-\sum_{i=1}^{n-1}G_i^* \overline{w}_i$.  
\end{lemma}

\begin{proof}
Equivalent to part 3 of the proof of Theorem 1 in
\cite{johnstone2018projective}.\ourqed 
\end{proof}

\preprintversion{Lemma 
\ref{lemStrat} can be intuitively understood as follows. If we define,
for all $k \geq 1$,
$$
\epsilon_k = \max_{i=1,\ldots,n}
\left\{\max\big\{\|y_i^k - w_i^k\|, \|G_i z^k - x_i^k\|\big\} \right\},
$$ 
then \eqref{eqAlltheMarbles} is equivalent to saying that $\epsilon_k\to 0$.
For all $k\geq 1$, we have $(x_i^k,y_i^k)\in \gra T_i$. If $\epsilon_k=0$,
then $w_i^k = y_i^k\in T_i x_i^k = T_i G_i z^k$ and since $\sum_{i=1}^n G_i^*
w_i^k = 0$, it follows that $(z^k,\bw^k)\in\calS$ and $z^k$ solves
\eqref{prob2}. Thus $\epsilon_k$ can be thought of as the ``residual" of the
algorithm which measures how far it is from finding a point in $\calS$ and a
solution to \eqref{prob2}. In finite dimension, it is straightforward to show
that if $\epsilon_k\to 0$, $(z^k,\bw^k)$ must converge to some element of
$\calS$. This can be done using Fej\'{e}r monotonicity \cite[Theorem
5.5]{bauschke2011convex} combined with the fact that the graph of a
maximal-monotone operator in a finite-dimensional Hilbert space is closed
\cite[Proposition 20.38]{bauschke2011convex}. However in the general Hilbert
space setting the proof is more delicate, since the graph of a
maximal-monotone operator is not in-general closed in the weak-to-weak
topology \cite[Example 20.39]{bauschke2011convex}. Nevertheless the overall
result was established in the general Hilbert space setting in part 3 of
Theorem 1 of \cite{johnstone2018projective}, which is itself an instance of
\cite[Proposition 2.4]{alotaibi2014solving} (see also \cite[Proposition
26.5]{bauschke2011convex}). An arguably more transparent proof can be found in
\cite{weakDong2018} (this proof is only for the case $n=2$, but it can be
extended).}

In order to establish \eqref{eqAlltheMarbles}, we start by establishing
certain contractive and ``ascent'' properties for the mapping $\calF$, and
also show that the backtracking procedure terminates finitely. Then, we prove
the boundedness of $x_i^k$ and $y_i^k$, in turn yielding the boundedness of
the gradients $\nabla
\varphi_k$ and hence the result that
$\limsup_{k\to\infty}\{\varphi_k(p^k)\}\leq 0$ by Lemma \ref{lemIsProject}.
Next we establish a ``Lyapunov-like" recursion for $\varphi_{i,k}(z^k,w_i^k)$, relating
$\varphi_{i,k}(z^k,w_i^k)$ to $\varphi_{i,k-1}(z^{k-1},w_i^{k-1})$. Eventually
this result will allow us to establish that $\liminf_k\varphi_k(p^k)\geq 0$
and hence that $\varphi_k(p^k)\to 0$, which will in turn allow an argument
that $y_i^k-w_i^k\to 0$. The proof that $G_i z^k - x_i^k\to 0$ will then
follow fairly elementary arguments.

The primary innovations of the upcoming proof are the ascent lemma and the way that it is used in Lemma
 \ref{lemFirstBiggy} to establish $\varphi_k(p^k)\to0$ and $y_i^k-w_i^k\to 0$.
 This technique is a significant deviation from previous analyses in the
 projective splitting family. In previous work, the strategy was to show that
 $\varphi_{i,k}(z^k,w_i^k)\geq C\max\{\|G_i z^k - x_i^k\|^2,\|y_i^k -
 w_i^k\|^2\}$ for a constant $C>0$, which may be combined with
 $\limsup\varphi_k(p^k)\leq 0$ to imply \eqref{eqAlltheMarbles}. In contrast,
 in the algorithm of this paper we cannot establish such a result and in fact
 $\varphi_{i,k}(z^k,w_i^k)$ may be negative. 
 \preprintversion{Instead, we relate $\varphi_k(p^k)$ to
 $\varphi_{k-1}(p^{k-1})$ to show that the separation improves at each
 iteration in a way which still leads to overall convergence.}
 

\preprintversion{\subsection{Some Basic Results}}
We begin by stating three elementary results on sequences,
which may be found in \cite{PolyakIntro}, and a basic, well known nonexpansivity property for forward steps with cocoercive operators. 
\begin{lemma}\label{lemBounddSeq} \emph{\cite[Lemma 1, Ch.~2]{PolyakIntro}}
	Suppose that $a_k\geq 0$ for all $k\geq 1$, $b\geq 0$, $0\leq\tau<1$, and 
	$a_{k+1}\leq\tau a_k+b$ for all $k\geq 1$.
	Then $\{a_k\}$ is a bounded sequence. 
\end{lemma}

\begin{lemma} \emph{\cite[Lemma 3, Ch.~2]{PolyakIntro}}
	Suppose that $a_k\geq 0, b_k\geq 0$ for all $k \geq 1$, $b_k\to 0$, and there is
	some $0\leq\tau<1$ such that $a_{k+1}\leq \tau a_k + b_k$ for all $k\geq 1$.
	Then $a_k\to 0$. 
	\label{lemSeqConverge}
\end{lemma} 

\begin{lemma}
	\label{lemBoundBelow}
	\myc{Suppose that $0\leq\tau<1$ and $\{r_k\},\{b_k\}$ are sequences in
	$\mathbb{R}$ with the properties $b_k\to 0$ and $r_{k+1}\geq \tau r_k +
	b_k$ for all $k\geq 1$.} Then $\lim\inf_{k\to\infty} \{ r_k \} \geq 0$.
\end{lemma}
\begin{proof}
Negating the assumed inequality yields $-r_{k+1}\leq \tau(-r_k)- b_k$.
Applying~\cite[Lemma 3, Ch.~2]{PolyakIntro} then yields $\lim\sup\{ -r_k\}\leq
0$.\ourqed
\end{proof}

\begin{lemma}\label{lemNonExp}
Suppose $B$ is $L^{-1}$-cocoercive and $0\leq\rho\leq2/L$. Then for all $x,y\in\dom(B)$
\begin{align}\label{eqNonexp}
\|x-y-\rho(Bx-By)\|\leq \|x-y\|.
\end{align}
\end{lemma}
\begin{proof}
Squaring the left hand side of \eqref{eqNonexp} yields
\begin{align*}
\|x-y-\rho(Bx-By)\|^2
&=
\|x-y\|^2 - 2\rho\langle x-y,Bx-By\rangle + \rho^2\|Bx-By\|^2
\\
&\leq 
\|x-y\|^2
-\frac{2\rho}{L}\|Bx-By\|^2
+ \rho^2\|Bx-By\|^2
\\
&\leq \|x-y\|^2.
\coapversion{\qquad\qquad\qquad\qquad\qquad\qquad\qquad\qquad\quad\ourqed}
\end{align*}
\end{proof}

%% file: contractive.tex
\subsection{A Contractive Result}
We begin the main proof with a result on the one-forward-step mapping: $\calF$ from Definition \ref{defUp}. The following lemma will ultimately be used to show that the iterates remain bounded. 

\begin{lemma}
\label{lemPreBounded}
Suppose $(x^+,y^+) = \calF_{\alpha,\rho}(z,x,w;A,B,G)$, where
$\calF_{\alpha,\rho}$ is given in Definition \ref{defUp}. Recall that $B$ is
$L^{-1}$-cocoercive. If $L=0$ or $\rho\leq 2(1-\alpha)/L$, then
\begin{align}\label{eqPreBounded}
\|x^+-\hat{\theta}\|
\leq 
(1-\alpha)\|x-\hat{\theta}\|+\alpha\|G z-\hat{\theta}\| 
      + \rho\left\|w - \hat{w}\right\|
\end{align}
for any $\hat{\theta}\in\dom(A)$ and $\hat{w}\in A \hat{\theta}+B\hat{\theta}$.
\end{lemma}
\begin{proof}
Select any $\hat{\theta}\in\dom(A)$ and $\hat{w} \in A\hat{\theta} +
B\hat{\theta}$.  Let $\hat{a} = \hat{w} - B\hat{\theta} \in A\hat{\theta}$.  It
follows immediately from \eqref{defprox2} that
\begin{align}\label{eqThetaHat}
\hat{\theta}=  J_{\rho A}(\hat{\theta} + \rho \hat{a}).
\end{align} 
Therefore, \eqref{eqF} and~\eqref{eqThetaHat} yield
\begin{align}
\|x^+ -\hat{\theta}\|
&=
\left\|J_{\rho A}
\big(
(1-\alpha)x+ \alpha G z - \rho (Bx-w)
\big)
-
J_{\rho A}(\hat{\theta} + \rho \hat{a})
\right\|
\nonumber\\
&\overset{(a)}{\leq}
\left\|
(1-\alpha)x+ \alpha G z - \rho (B x-w)
-
\hat{\theta} - \rho \hat{a}
\right\|
\nonumber\\ 
&\overset{(b)}{=}
\left\|
(1-\alpha)
\left(x-\hat{\theta}
-\frac{\rho}{1-\alpha}
\left(
B x
-
B \hat{\theta}
\right)
\right)
+\alpha (G z-\hat{\theta}) 
+ \rho\left(w - \hat{a}-B \hat{\theta}\right)
\right\|
\nonumber\\
&\overset{(c)}{\leq}\label{eqTakeLong}
(1-\alpha)
\left\|
x-\hat{\theta}
-\frac{\rho}{1-\alpha}
\left(
Bx
-
B \hat{\theta}
\right)
\right\|
+\alpha\|G z-\hat{\theta}\|
+ \rho\left\|w - (\hat{a}+B \hat{\theta})\right\|
\\\nonumber 
&\overset{(d)}{\leq}
(1-\alpha)\|x-\hat{\theta}\|+\alpha\|G z-\hat{\theta}\| 
   + \rho\left\|w - \hat{w}\right\|.
\end{align}
To obtain (a), one uses the nonexpansivity of the resolvent
\cite[Prop.~23.8(ii)]{bauschke2011convex}. To obtain (b), one regroups terms and adds and
subtracts $B\hat{\theta}$. Then (c) follows from the triangle inequality.
Finally we consider (d): If $L>0$, apply Lemma \ref{lemNonExp} to the first term on the right-hand side of \eqref{eqTakeLong}
with the stepsize $\rho/(1-\alpha)$ which by assumption satisfies
\begin{align*}
\frac{\rho}{1-\alpha}\leq \frac{2}{L}
\end{align*} 
by Assumption \ref{assStep}. 
Alternatively, if $L=0$, implying that $B$ is
a constant-valued operator, then $B x=B\hat{\theta}$ and (d) is
just an equality.
\end{proof}


%% file: ascent.tex
\preprintversion{We now prove the key ``ascent lemma''. It shows that, while the
one-forward-step update is not guaranteed to find a separating hyperplane at
each iteration, it does make a certain kind of progress toward separation.}
\begin{lemma}\label{lemAscent0}
Suppose $(x^+,y^+) = \calF_{\alpha,\rho}(z,x,w;A,B,G)$, where
$\calF_{\alpha,\rho}$ is given in Definition \ref{defUp}. Recall $B$ is
$L^{-1}$-cocoercive. Let $y\in Ax+Bx$ and define $\varphi\triangleq \langle G
z-x,y-w\rangle$. Further, define $\varphi^+\triangleq \langle G z - x^+,y^+ -
w\rangle$, $t$ as in \eqref{eqF}, and $\hat{y}\triangleq \rho^{-1}(t-x^+)+B
x$. If \myc{$\alpha \in (0,1]$ and} $\rho\leq2(1-\alpha)/L$ whenever $L>0$, then
\begin{align}
\varphi^+
&\geq 
\frac{\rho}{2\alpha}
\left(
\|y^+ - w\|^2
+
\alpha \|\hat{y} - w\|^2
\right)
+
(1-\alpha)
\left(
\varphi
-\frac{\rho}{2\alpha}\|y-w\|^2
\right).
\label{eqLong0}
\end{align}
\end{lemma} 

\begin{proof}
 Since $y\in Ax+Bx$, there exists $a\in Ax$ such that $y=a+Bx$. Let $a^+
 \triangleq \rho^{-1}(t-x^+)$. Note by \eqref{defprox2} that $a^+\in A x^+$.
 With this notation, $\hat{y} = a^++B x$.
 
 We may write the $x^+$-update in \eqref{eqF} as
 \begin{align*}
 x^++\rho a^+ = (1-\alpha)x+\alpha Gz - \rho (B x - w)
 \end{align*}
 which rearranges to
 \begin{align*}
 x^+ = (1-\alpha)x+\alpha Gz - \rho (\hat{y} - w)
 \implies 
- x^+ = -\alpha G z - (1-\alpha) x + \rho(\hat{y} - w).
\end{align*}
Adding $G z$ to both sides yields
\begin{align}\label{eqSecondSub}
G z - x^+ = (1-\alpha)(G z - x)+\rho(\hat{y} - w).
\end{align}
Substituting this equation into the definition of $\varphi^+$
yields
\begin{align}
\varphi^+
&=
\langle G z - x^+,y^+ - w\rangle 
\nonumber\\\nonumber
&=
\big\langle 
(1-\alpha)(G z - x) + \rho(\hat{y} - w),
y^+ - w\big\rangle  
\\\nonumber
&=
(1-\alpha)\langle G z - x, y^+ - w\rangle  
+
\rho
\langle \hat{y} - w, y^+ - w \rangle 
\\\nonumber
&=
(1-\alpha)\langle G z - x, y - w\rangle  
+
(1-\alpha)\langle G z - x, y^+ - y\rangle  
+ \rho
\langle \hat{y} - w, y^+ - w \rangle
\\\label{eqEnno}  
&=
(1-\alpha)\varphi
+
(1-\alpha)\langle G z - x, y^+ - y\rangle  
+
\rho
\langle \hat{y} - w, y^+ - w \rangle.
\end{align}
We now focus on the second term 
in
\eqref{eqEnno}.
Assume for now that $L>0$ (we will deal with the $L=0$ case below).
We write
\begin{align}
\langle G z - x,y^+ - y\rangle 
&=
\langle x^+ - x,y^+ - y\rangle 
+
\langle G z - x^+,y^+ - y\rangle 
\nonumber\\
&=
\langle x^+ - x,a^+ - a\rangle 
+
\langle x^+ - x,B x^+ - Bx\rangle 
\coap
+
\langle G z - x^+,y^+ - y\rangle 	
\label{eqDeduce}\\
&\geq 
L^{-1}\|B x^+ - B x\|^2
+
\langle G z - x^+,y^+ - y\rangle 
\nonumber\\\nonumber 
&=
L^{-1}\|B x^+ - B x\|^2
+
\langle G z - x^+,y^+ - w\rangle 
\coap
+
\langle G z - x^+,w - y\rangle 
\nonumber\\
&=
L^{-1}\|B x^+ - B x\|^2
+
\varphi^+
+
\langle G z - x^+,w - y\rangle.
\label{eqSecondTerm}
\end{align}
\myc{To derive \eqref{eqDeduce} we substituted $(y^+,y) =(a^+ + B x^+,a+Bx)$} 
and \myc{for the following} inequality we used the monotonicity of $A$ and
$L^{-1}$-cocoercivity of $B$ (recall that $a\in A x$ and $a^+\in A x^+$). 
Substituting the resulting inequality
back into
\eqref{eqEnno} 
yields 
\myc{ 
\begin{align*}
\varphi^+
 &= 
 (1-\alpha)\varphi
 +
 (1-\alpha)\langle G z - x, y^+ - y\rangle  
 +
 \rho
 \langle \hat{y} - w, y^+ - w \rangle
\\
&\geq 
 (1-\alpha)\varphi
 +(1-\alpha)
 \left(
 L^{-1}\|B x^+ - B x\|^2
 +
 \varphi^+
 +
 \langle G z - x^+,w - y\rangle
 \right)
 \\
 &
 \qquad
 +
 \rho
 \langle \hat{y} - w, y^+ - w \rangle.
\end{align*}
Subtracting $(1-\alpha)\varphi^+$ from both sides of the above inequality produces}
\begin{align}
\alpha\varphi^+
&\geq 
(1-\alpha)
\left(
\varphi
+
L^{-1}\|B x^+ - B x\|^2
+
\langle G z - x^+,w - y\rangle
\right)
\coap 
+
\rho
\langle \hat{y} - w, y^+ - w \rangle.
\label{eqLanding}
\end{align}
Using \eqref{eqSecondSub} once again, this time to the third term on the
right-hand side of \eqref{eqLanding}, we write
\begin{align}
\langle G z - x^+,w - y\rangle
&=
\big\langle(1-\alpha)(G z - x)+\rho(\hat{y}-w),
w - y\big\rangle
\nonumber\\
&=
(1-\alpha) \langle G z - x, w - y\rangle 
+
\rho
\langle\hat{y}-w,w - y\rangle 
\nonumber\\\label{eqHoa}
&=
(\alpha-1)\varphi
-
\rho
\langle\hat{y}-w,y-w\rangle.
\end{align}
Substituting this equation back into \eqref{eqLanding} yields
\begin{align}
\alpha\varphi^+
&\geq 
(1-\alpha)\left(
\alpha\varphi
+
L^{-1}\|B x^+ - B x\|^2
-
\rho
\langle\hat{y}-w,y - w\rangle
\right)
\coap 
+
\rho
\langle \hat{y} - w, y^+ - w \rangle.\label{eqPaul}
\end{align}
We next use the identity $\langle x_1,x_2\rangle =
\frac{1}{2}\|x_1\|^2+\frac{1}{2}\|x_2\|^2-\frac{1}{2}\|x_1-x_2\|^2$ on both inner
products in \eqref{eqPaul}, as follows:
\begin{align}
\langle \hat{y} - w,y-w\rangle
&=
\frac{1}{2}
\left(
\|\hat{y} - w\|^2
+
\|y-w\|^2
-
\|\hat{y} - y\|^2
\right)
\nonumber\\\label{eqGetas}
&=
\frac{1}{2}
\left(
\|\hat{y} - w\|^2
+
\|y-w\|^2
-
\|a^+ - a\|^2
\right)
\end{align}
and
\begin{align}
\langle \hat{y} - w,y^+ - w\rangle
&=
\frac{1}{2}
\left(
\|\hat{y} - w\|^2
+
\|y^+ - w\|^2
-
\|\hat{y} - y^+\|^2
\right)
\nonumber\\\label{eqGetbs}
&=
\frac{1}{2}
\left(
\|\hat{y} - w\|^2
+
\|y^+ - w\|^2
-
\|B x^+ - B x\|^2
\right).
\end{align}
Here we have used the identities
\begin{align*}
\hat{y} - y
&= a^+ + B x - (a + B x)=a^+ - a
\\
\hat{y} - y^+
&= a^+ + Bx - (a^+ + B x^+)=Bx - Bx^+.
\end{align*}
Using \eqref{eqGetas}--\eqref{eqGetbs} in \eqref{eqPaul} \myc{yields
\begin{align}
\alpha\varphi^+
&\geq 
(1-\alpha)\left(
\alpha\varphi
+
L^{-1}\|B x^+ - B x\|^2
-
\rho
\langle\hat{y}-w,y - w\rangle
\right)
\coap 
+
\rho
\langle \hat{y} - w, y^+ - w \rangle
\nonumber\\
&=
(1-\alpha)
\left(
\alpha\varphi +L^{-1}\|B x^+ - Bx\|^2
\right)
\nonumber\\
&\qquad
-
\frac{\rho(1-\alpha)}{2}
\left(
\|\hat{y}-w\|^2 + \|y - w\|^2 - \|a^+-a\|^2
\right)
\nonumber\\
&\qquad 
+\frac{\rho}{2}
\left(
\|\hat{y}-w\|^2+\|y^+-w\|^2-\|Bx^+-Bx\|^2
\right)
\nonumber\\
&= 
(1-\alpha)
\left(
\alpha\varphi
-\frac{\rho}{2}\|y-w\|^2
+
\frac{\rho}{2}
\|a^+ - a\|^2
\right)
\coap 
+
\left(\frac{1-\alpha}{L} - \frac{\rho}{2}\right)
\|B x^+ - B x\|^2
\paoc 
+\frac{\rho}{2}
\left(
\|y^+ - w\|^2
+
\alpha \|\hat{y} - w\|^2
\right).\nonumber 
\end{align}
Consider this last expression: since $\alpha \leq 1$, the coefficient
$(1-\alpha)\rho/2$ multiplying $\|a^+ - a\|^2$ is nonnegative. Furthermore,
since $\rho\leq2(1-\alpha)/L$, the coefficient multiplying $\|B x^+ - B x\|^2$
is positive. Therefore we may drop these two terms from the above inequality
and divide by $\alpha$ to obtain \eqref{eqLong0}.}

Finally, we deal with the case in which $L = 0$, which implies that $B x =
v$ for some $v\in\calH$ for all $x\in\calH$. The main difference is
that the $\|B x^+ - Bx\|^2$ terms are no longer present since $B x^+ =
B x$. The analysis is the same up to \eqref{eqEnno}.
In this case $B x^+ = v$ so instead of
\eqref{eqSecondTerm} we may deduce from \eqref{eqDeduce}  that
\begin{align*}
\langle G z - x,y^+ - y\rangle 
\; \geq \;
\varphi^+
+
\langle G z - x^+,w - y\rangle. 
\end{align*}  
Since $B x^+=B x=v$ is constant we also have that 
$$
\hat{y} = a^+ + B x=a^+ + v = a^+ + B x^+ = y^+
$$ 
Thus, instead of \eqref{eqLanding} in this case we have the
simpler inequality
\begin{align}
\alpha\varphi^+
&\geq 
(1-\alpha)\left(\varphi
+
\langle G z - x^+,w - y\rangle
\right)
+
\rho
\|y^+ - w\|^2
. \label{eqNew}
\end{align}
The term $\langle G z - x^+,w - y\rangle$ in \eqref{eqNew}
is dealt with just as in \eqref{eqLanding}, by substitution
of \eqref{eqSecondSub}. This step now leads via \eqref{eqHoa} to
\begin{align*}
\alpha\varphi^+
&\geq 
\alpha(1-\alpha)\varphi
-\rho(1-\alpha)\langle y^+ - w,y-w\rangle
+
\rho
\|y^+ - w\|^2.
\end{align*}
Once again using $\langle x_1,x_2\rangle = \frac{1}{2}\|x_1\|^2+\frac{1}{2}\|x_2\|^2 - \frac{1}{2}\|x_1-x_2\|^2$ on the second term on the r.h.s.~above yields
\begin{align*}
\alpha\varphi^+
&\geq
\alpha(1-\alpha)\varphi
+
\rho
\|y^+ - w\|^2
\coap
-
\frac{\rho(1-\alpha)}{2} 
\left(
\|y^+ - w\|^2
+
\|y - w\|^2
-
\|y^+ - y\|^2
\right).
\end{align*}
We can lower-bound the  $\|y^+ - y\|^2$ term by $0$. Dividing through by $\alpha$ and rearranging, we obtain 
\begin{align*}
\varphi^+
&\geq 
\frac{\rho(1+\alpha)}{2\alpha}
\|y^+ - w\|^2
+
(1-\alpha)
\left(
\varphi
-\frac{\rho}{2\alpha}\|y-w\|^2
\right).
\end{align*}
Since $y^+=\hat{y}$ in the $L=0$ case, this is equivalent to \eqref{eqLong0}.\ourqed
\end{proof}

%% file: bounded.tex
\subsection{Finite Termination of Backtracking}\label{secFinite}
In all the following lemmas in sections \ref{secFinite} and \ref{secBounded}
regarding algorithms \ref{AlgfullWithBT}--\ref{AlgProjUpdate}, assumptions
\ref{assMono} and \ref{assStep} are in effect and will not be explicitly
stated in each lemma. We start by proving that \texttt{backTrack} terminates
in a finite number of iterations, and that the stepsizes it returns are
bounded away from $0$.
\begin{lemma}\label{lemFinite}
For $i\in\mathcal{B}$, Algorithm \ref{AlgBackTrack} terminates in a finite
number of iterations for all $k\geq 1$. There exists $\urho_i>0$ such that
$\rho_i^k\geq\urho_i$ for all $k\geq 1$, where $\rho_i^k$ is the stepsize
returned by Algorithm \ref{AlgBackTrack} on line \ref{lineDoBT}. Furthermore $\rho_i^k\leq\hat{\rho}$ for all $k\geq 1$. 
\end{lemma}
\begin{proof}

Assume we are at iteration $k\geq 1$ in Algorithm \ref{AlgfullWithBT} and
\texttt{backTrack} has been called through line \ref{lineDoBT} for some
$i\in\mathcal{B}$. The internal variables within \texttt{backTrack} are
defined in terms of the variables passed from Algorithm \ref{AlgfullWithBT} as
follows: $z=z^k$,  $x = x_i^{k-1}$, $w=w_i^k$, $y=y_i^{k-1}$, $\rho = \rho_i^{k-1}$ and $\eta = \eta_i^{k-1}$. Furthermore
$\alpha=\alpha_i$,
$\hat{\theta}=\hat{\theta}_i$, $\hat{w}=\hat{w}_i$, $A=A_i$, $B=B_i$, and
$G=G_i$. 
The calculation on line \ref{lineGetPhi} of Algorithm \ref{AlgBackTrack} yields $\varphi =
\varphi_{i,k-1}(z^k,w_i^k)$.
In the following argument, we mostly refer to the internal name of
the variables within \texttt{backTrack} without explicitly making the above
substitutions. With that in mind, let $L=L_i$ be the cocoercivity constant of
$B=B_i$.

Recall that $\tilde{\rho}^{(i,k)}$ is the initial trial stepsize $\tilde{\rho}_1$ chosen on line \ref{lineChooseRho} of \texttt{backTrack}.
We must establish that the interval on line \ref{lineChooseRho} is always nonempty and so a valid initial stepsize can be chosen. Since $\eta\alpha\geq 0$, this will be true if  $\hat{\rho}\geq\rho=\rho_i^{k-1}$, which we will prove by induction. 
Note that by Assumption \ref{assStep}, $\hat{\rho}\geq \rho_i^0$ for all $i\in\mathcal{B}$. Therefore for $k=1$, $\hat{\rho}\geq \rho=\rho_i^0$. 
We will prove the induction step below.

Observe that backtracking terminates via line \ref{lineBTreturn} if two conditions are met. The first condition,
\begin{align}\label{eqCond1}
\|\tilde{x}_{j} - \hat{\theta}\|\leq 	(1-\alpha)\|x-\hat{\theta}\|
+
\alpha\| G z- \hat{\theta}\|+\tilde{\rho}_{j}\|w - \hat{w}\|,
\end{align}
is identical to \eqref{eqPreBounded} of Lemma~\ref{lemPreBounded}, with
$\tilde{x}_j$ and $\tilde{\rho}_j$ respectively in place of $x^+$ and $\rho$.
The initialization step of Algorithm~\ref{AlgBackTrack} provides us with
$\hat{w}\in A\hat{\theta}+B\hat{\theta}$ for some $\hat{\theta}\in
\dom(A)$. Furthermore, since
\begin{align*}
(\tilde{x}_j,\tilde{y}_j)
=
\calF_{\alpha,\tilde{\rho}_j}(z,x,w;A,B,G),
\end{align*}
the findings of Lemma \ref{lemPreBounded} may be applied. In particular, if $L
> 0$ and $\tilde{\rho}_j\leq 2(1-\alpha)/L$, then \eqref{eqCond1} will be met.
Alternatively, if $L=0$, \eqref{eqCond1} will hold for any value of the
stepsize $\tilde{\rho}_j>0$.

Next, consider the second termination condition,
\begin{align}\label{eqCond2}
\varphi_j^+
\geq			 
\frac{\tilde{\rho}_{j}}{2\alpha}
\left(
\|\tilde{y}_{j} - w\|^2
+
\alpha 
\|\hat{y}_{j} - w\|^2
\right)+ (1-\alpha)
\left(
\varphi
-
\frac{\tilde{\rho}_{j}}{2\alpha}
\|y - w\|^2
\right).
\end{align}
This relation is identical to \eqref{eqLong0} of Lemma \ref{lemAscent0}, with
$(\tilde{y}_j,\hat{y}_j,\tilde{\rho}_j)$ in place of $(y^+,\hat{y},\rho)$.
However, to apply the lemma we must show that $y=y_i^{k-1}\in A x_i^{k-1}+B
x_i^{k-1}=A x+B x$. We will also prove this by induction.

For $k=1$, $y=y_i^{k-1}\in A x_i^{k-1}+B x_i^{k-1}=A x+B x$ holds by the
initialization step of Algorithm~\ref{AlgfullWithBT}. Now assume that 
at iteration $k\geq 2$ it holds that 
$y=y_i^{k-1}\in A x_i^{k-1}+B x_i^{k-1}=A x+B x$
and furthermore that 
$\hat{\rho}\geq\rho=\rho_i^{k-1}$, therefore the interval on line \ref{lineChooseRho}  is nonempty. We may then apply the findings of Lemma
\ref{lemAscent0} to conclude that if $L > 0$ and $\tilde{\rho}_j\leq
2(1-\alpha)/L$, then condition \eqref{eqCond2} is satisfied.  Or, if $L=0$,
condition~\eqref{eqCond2} is satisfied for any $\tilde{\rho}_j>0$.

Combining the above observations, we conclude that if $L > 0$ and
$\tilde{\rho}_j\leq 2(1-\alpha)/L$, backtracking will terminate for that iteration $j$ of \texttt{backTrack} via line \ref{lineBTreturn}. Or, if $L=0$, it will terminate in the
first iteration of \texttt{backTrack}. The stepsize decrement condition on line \ref{lineStepDec} of
the backtracking procedure implies that $\tilde{\rho}_j\leq 2(1-\alpha)/L$
will eventually hold for large enough $j$, and hence that the two backtracking
termination conditions must eventually hold.

Let $j^*\geq 1$ be the iteration at which
backtracking terminates when called for operator $i$ at iteration $k$ of
Algorithm~\ref{AlgfullWithBT}. For the pair $(x_i^k,y_i^k)$ returned by
\texttt{backTrack} on line \ref{lineDoBT} of Algorithm \ref{AlgfullWithBT}, we
may write
\begin{align*}
(x_i^k,y_i^k)&=(\tilde{x}_{j^*},\tilde{y}_{j^*}) 
\coapnquad 
= \calF_{\alpha,\tilde{\rho}_{j^*}}(z,x,w;A,B,G)=
\calF_{\alpha_i^k,\rho_i^k}(z^k,x_i^{k-1},w_i^k;A_i,B_i,G_i).
\end{align*} 
Thus, by the definition of $\calF$ in \eqref{eqF}, $y_i^k\in A_i x_i^k + B_i
x_i^k$. Therefore, induction establishes that $y_i^k\in A_i x_i^k + B_i x_i^k$
holds for all $k\geq 1$. 

Now the returned stepsize must satisfy $\rho_i^k=\tilde{\rho}_{j^*}\leq \tilde{\rho}^{(i,k)}\leq \hat{\rho}$. In the next iteration, $\rho=\rho_i^{k}\leq \hat{\rho}$. Thus we have also established by induction that $\hat{\rho}\geq \rho=\rho_i^k$ and therefore  that the interval on line \ref{lineChooseRho} is nonempty for all iterations $k\geq 1$.
Finally, we now also infer by induction that \texttt{backTrack}
terminates in a finite number of iterations for all $k\geq 1$ and
$i\in\mathcal{B}$.

 Now $\tilde{\rho}^{(i,k)}$ must be chosen in the range $$\tilde{\rho}^{(i,k)}\in \left[\rho_i^{k-1},
 \min\left\{(1+\alpha_i\eta_i^{k-1})\rho_i^{k-1},\hat{\rho}\right\}\right].$$
Since we have established that this interval remains nonempty, it holds trivially that
$\tilde{\rho}^{(i,k)}\geq \rho_i^{k-1}$.
For all $k\geq 1$ and $i\in\mathcal{B}$, the returned stepsize
$\rho_i^k=\tilde{\rho}_{j^*}$ must satisfy
\begin{align}
&(\forall\,i: L_i>0):\quad 
\rho_i^k
\geq 
\min\left\{
\tilde{\rho}^{(i,k)},
\frac{2\delta(1-\alpha_i)}{L_i}
\right\}
\label{eqHoA1}\\ 
&(\forall\,i: L_i=0):\quad \rho_i^k = \tilde{\rho}^{(i,k)}.
\nonumber 
\end{align}
Therefore for all $k\geq 1$ and
all $i\in\mathcal{B}$ such that $L_i > 0$, one has 
\begin{align*}
\rho_i^k
\geq 
\min\left\{
\rho_i^{k-1},
\frac{2\delta(1-\alpha_i)}{L_i}
\right\}
&\geq 
\min\left\{
\rho_i^{1},
\frac{2\delta(1-\alpha_i)}{L_i}
\right\}
\\
&\geq 
\min\left\{
\rho_i^{0},
\frac{2\delta(1-\alpha_i)}{L_i}
\right\}\triangleq \urho_i>0,
\end{align*}
where the first inequality uses \eqref{eqHoA1} and $\tilde{\rho}^{(i,k)}\geq\rho_i^{k-1}$, the second inequality recurses, and the final inequality is just 
\eqref{eqHoA1} for $k=1$. If $L_i=0$, the argument is simply
$$
\rho_i^k = \tilde{\rho}^{(i,k)}\geq \rho_i^{k-1}=\tilde{\rho}^{(i,k-1)}\geq\ldots\geq  
\rho_i^1  =\tilde{\rho}^{(i,1)} = 
\rho_i^0\triangleq \urho_i>0. 
\vspace{-5ex}
$$\ourqed
\end{proof}

\subsection{Boundedness Results and their Direct Consequences}\label{secBounded}

\begin{lemma}\label{lemBounded}
\myc{	
	For all $i=1,\ldots,n$, the sequences
	$\{x_i^k\}$ and $\{y_i^k\}$ are bounded. 
}
\end{lemma}

\begin{proof}
	To prove this, we first establish that for $i=1,\ldots,n$ and $k\geq 1$ 
	\begin{align}\label{eqPre2Bounded}
	\|x_i^k-\hat{\theta}_i\|
	\leq 
	(1-\alpha_i)\|x_i^{k-1}-\hat{\theta}_i\|+\alpha_i\|G_i z^k-\hat{\theta}_i\| + \hat{\rho}\left\|w_i^k - \hat{w}_i\right\|
	\end{align}	
	For $i\in\mathcal{B}$, Lemma \ref{lemFinite} establishes that
	\texttt{backTrack} terminates for finite $j\geq1$ for all $k\geq 1$. For
	fixed $k\geq 1$ and $i\in\mathcal{B}$, let $j^*\geq 1$ be the iteration of
	\texttt{backTrack} that terminates. At termination, the following
	condition is satisfied via line \ref{lineC1}:
\begin{align*} 
\|\tilde{x}_{j^*} - \hat{\theta}\|\leq 	(1-\alpha)\|x-\hat{\theta}\|
+
\alpha\| G z- \hat{\theta}\|+\tilde{\rho}_{j^*}\|w - \hat{w}\|.
\end{align*}
Into this inequality, now substitute in the following variables from
Algorithm \ref{AlgfullWithBT}, as passed to and from \texttt{backTrack}:
$x_i^k = \tilde{x}_{j^*}$, $\hat{\theta}_i=\hat{\theta}$, $\alpha_i=\alpha$,
$x_i^{k-1}=x$, $G_i=G$, $z^k = z$, $\rho_i^k = \tilde{\rho}_{j^*}$, $w_i^k =
w$, and $\hat{w}_i = w$.  Further noting that $\rho_i^k\leq \hat{\rho}$, the result is~\eqref{eqPre2Bounded}.

For $i\notin\mathcal{B}$, we note that line \ref{lineDoNoBT} of Algorithm
\ref{AlgfullWithBT} reads as
\begin{align*}
(x_i^k,y_i^k) = \calF_{\alpha_i, \rho_i}(z^k, x_i^{k-1}, w_i^k; A_i, B_i, G_i)
\end{align*}
and since Assumption \ref{assStep} holds, we may apply Lemma \ref{lemPreBounded}. Further noting that by Assumption \ref{assStep} $\rho_i\leq\hat{\rho}$ we arrive at
yield \eqref{eqPre2Bounded}.

Since $\{z^k\}$, and $\{w_i^k\}$ are
bounded by Lemma~\ref{lemIsProject} and $\|G_i\|$ is bounded by
Assumption~\ref{AssMonoProb}, boundedness of $\{x_i^k\}$ now follows by applying
Lemma~\ref{lemBounddSeq} with $\tau=1-\alpha_i<1$ to \eqref{eqPre2Bounded}.

Next, boundedness of $B_i x_i^k$ follows from the continuity of $B_i$. Since
Lemma \ref{lemFinite} established that \texttt{\backTrack} terminates in a
finite number of iterations we have for any $k \geq 2$ that
\begin{align*}
(x_i^k,y_i^k) = \calF_{\alpha_i^k, \rho_i^k}(z^k, x_i^{k-1}, w_i^k; A_i, B_i, G_i)
\end{align*}
where for $i\notin \mathcal{B}$ $\rho_i^k \triangleq \rho_i$.
Expanding the $y^+$-update in the definition of $\calF$ in \eqref{eqF}, we may write
\begin{align*}
y_i^k = (\rho_i^k)^{-1} 
\left(
(1-\alpha_i)x_i^{k-1}+\alpha_i G_i z^k -\rho_i^k(B_i x_i^{k-1} - w_i^k)
-x_i^k
\right)
+ B x_i^k.
\end{align*}
Since $G_i$, $z^k$, and $w_i^k$ are bounded, for $i\in\mathcal{B}$ 
$\rho_i^k\leq\hat{\rho}$, and $\rho_i^k\geq \urho_i$ (using Lemma \ref{lemFinite}
for $i\in\mathcal{B}$), and for $i\notin\mathcal{B}$ $\rho_i^k=\rho_i$ is constant, we conclude that $y_i^k$ remains bounded.\ourqed
\end{proof}
With $\{x_i^k\}$ and $\{y_i^k\}$ bounded for all $i=1,\ldots,n$, the
boundedness of $\nabla\varphi_k$ follows immediately:
\begin{lemma}\label{lemBoundedGrad}
The sequence $\{\nabla\varphi_k\}$ is bounded. If Algorithm
\ref{AlgfullWithBT} never terminates via line \ref{lineReturn},
$\limsup_{k\to\infty}\varphi_k(p^k)\leq 0$.
\end{lemma} 
\begin{proof}
By Lemma \ref{LemGradAffine}, $\nabla_z\varphi_k = \sum_{i=1}^n G_i^* y_i^k$,
which is bounded since each $G_i$ is bounded by assumption and each
$\{y_i^k\}$ is bounded by Lemma \ref{lemBounded}. Furthermore,
$\nabla_{w_i}\varphi_k= x_i^k - G_i x_n^k$  is bounded using the same two
lemmas. That $\limsup_{k\to\infty}\varphi_k(p^k)\leq 0$ then immediately
follows from Lemma \ref{lemIsProject}(3).\ourqed
\end{proof}

Using the boundedness of $\{x_i^k\}$ and $\{y_i^k\}$, we can next derive the
following simple bound relating $\varphi_{i,k-1}(z^k,w_i^k)$ to
$\varphi_{i,k-1}(z^{k-1},w_i^{k-1})$:

\begin{lemma}\label{lemLBPhi}
There exists $M_1,M_2\geq 0$ such that for all $k\geq 2$ and $i=1,\ldots,n$,
\begin{align*}
\varphi_{i,k-1}(z^k,w_i^k)
&\geq 
\varphi_{i,k-1}(z^{k-1},w_i^{k-1})
-
M_1\|w_i^k-w_i^{k-1}\|
\coap 
-M_2\|G_i\|\| z^k -  z^{k-1}\|.
\end{align*}
\end{lemma}
\begin{proof}
For each $i\in\{1,\ldots,n\}$, let $M_{1,i},M_{2,i}\geq 0$ be
respective bounds on $\big\{\|G_i z^{k-1}-x_i^{k-1}\|\big\}$ and
$\big\{\|y_i^{k-1}-w_i^k\|\big\}$, which must exist by
Lemma~\ref{lemIsProject}, the boundedness of $\{x_i^k\}$ and $\{y_i^k\}$, and
the boundedness of $G_i$.  Let $M_1 = \max_{i=1,\ldots,m} \{ M_{1,i}\}$ and
$M_2 = \max_{i=1,\ldots,m} \{ M_{2,i}\}$.  Then, for any $k \geq 2$ and
$i\in\{1,\ldots,n\}$, we may write
\begin{align*}
\coapAmp\varphi_{i,k-1}(z^k,w_i^k)
\coapDD
&=
\langle G_i z^k - x_i^{k-1},y_i^{k-1}-w_i^k\rangle 
\\
&=
\langle G_i z^{k-1} - x_i^{k-1},y_i^{k-1}-w_i^k\rangle 
+
\langle G_i z^k - G_i z^{k-1},y_i^{k-1}-w_i^k\rangle 
\\
&=
\langle G_i z^{k-1} - x_i^{k-1},y_i^{k-1}-w_i^{k-1}\rangle 
\coap 
+
\langle G_i z^{k-1} - x_i^{k-1},w_i^{k-1}-w_i^k\rangle 
\paoc 
+\langle G_i z^k - G_i z^{k-1},y_i^{k-1}-w_i^k\rangle 
\\
&\geq 
\varphi_{i,k-1}(z^{k-1},w_i^{k-1})
-
M_1\|w_i^k-w_i^{k-1}\|
-M_2\|G_i\|\| z^k -  z^{k-1}\|,
\end{align*}
where the last step uses the Cauchy-Schwarz inequality and the definitions of
$M_1$ and $M_2$.\ourqed
\end{proof}

%% file: mainProof-part2.tex
\subsection{A Lyapunov-Like Recursion for the Hyperplane}\label{secLyap}
 We now establish a Lyapunov-like recursion for the hyperplane. For this
 purpose, we need two more definitions.

\begin{definition}\label{defYhat}
	For all $k\geq 1$, since Lemma \ref{lemFinite} establishes that Algorithm \ref{AlgBackTrack} terminates in a finite number of iterations, we may write for $i=1,\ldots,n$:
	\begin{align}\nonumber
	(x_i^k,y_i^k) = \calF_{\myc{\alpha_i}, \rho_i^k}(z^k, x_i^{k-1}, w_i^k; A_i, B_i, G_i)
	\end{align}
	where for $i\notin\mathcal{B}$ $\rho_i^k=\rho_i$ are actually fixed. 
	Using \eqref{defprox2} and the $x^+$-update in \eqref{eqF}, there exists $a_i^k\in A_i x_i^k$ such that
	\begin{align}\nonumber
	x_i^k+\rho_i^k a_i^k = (1-\alpha_i)x_i^{k-1}+\alpha_i G_i z^k -\rho_i^k (B_i x_i^{k-1} - w_i^k).
	\end{align}
	Define $\hat{y}_i^k\triangleq a_i^k+B_i x_i^{k-1}$.	
\end{definition}
\myc{
\begin{definition}
	For $i\notin\mathcal{B}$ we will use $\rho_i^k\triangleq \rho_i$, even though these stepsizes are fixed, so that we can use the same statements as for $i\in\mathcal{B}$. Similarly we will use $\urho_i\triangleq\rho_i$ for $i\notin\mathcal{B}$. 
\end{definition}

\begin{lemma}\label{lemAv2step}
	For all $k\geq 1$, and $i=1,\ldots,n$
	\begin{align}\label{eqStep2}
	\frac{\rho_i^{k+1}}{\alpha_i}\|y_i^k - w_i^k\|^2
	\leq 
	\frac{\rho_i^k}{\alpha_i}
	\left(
	\|y_i^k - w_i^k\|^2
	+
	\alpha_i
	\|\hat{y}_i^k - w_i^k\|^2
	\right).
	\end{align}	
\end{lemma}
\begin{proof}
	For $i\in\mathcal{B}$, 
recall that $\tilde{\rho}^{(i,k)}$ is the initial trial stepsize chose on line \ref{lineChooseRho} of \texttt{backTrack} at iteration $k$ for some $i\in\mathcal{B}$. 
The condition on line \ref{lineChooseRho} of \texttt{backTrack} guarantees that
\begin{align*}
\tilde{\rho}^{(i,k+1)}
\leq 
\rho_i^k
\left(
1
+
\alpha_i
\frac{ 
\|\hat{y}_i^k - w_i^k\|^2
}{
\|y_i^k - w_i^k\|^2
}
\right).
\end{align*}
Multiplying through by $\alpha_i^{-1}\|y_i^k - w_i^k\|^2$ and noting that $\rho_i^{k+1}\leq \tilde{\rho}^{(i,k+1)}$ proves the lemma. 

For $i\notin\mathcal{B}$ the expression holds trivially because $\rho_i^{k+1}=\rho_i^k=\rho_i$.  	
\ourqed 
\end{proof}
}

\begin{lemma}\label{lemAscent1}
	For all $k\geq 2$ and $i=1,\ldots,n$,
	\begin{multline}
	\varphi_{i,k}(z^k,w_i^k)
	-
	\frac{\rho_i^k}{2\myc{\alpha_i}}
	\left(
	\|y_i^k - w_i^k\|^2+\myc{\alpha_i}\|\hat{y}_i^k - w_i^k\|^2
	\right) 
	\\
	\geq 
	(1-\myc{\alpha_i})
	\left(
	\varphi_{i,k-1}(z^k,w_i^k)
	-
	\frac{\rho_i^k}{2\myc{\alpha_i}}
	\|y_i^{k-1} - w_i^k\|^2
	\right) \label{eqAprelim}
	\end{multline}
	and
	\begin{align}
	\varphi_{i,k}(z^k,w_i^k)
	-&
	\frac{\rho_i^{k+1}}{2\myc{\alpha_i}}
	\|y_i^k - w_i^k\|^2	
	\coap 
	\geq
	(1-\myc{\alpha_i})
	\left(
	\varphi_{i,k-1}(z^k,w_i^k)
	-
	\frac{\rho_i^k}{2\myc{\alpha_i}}
	\|y_i^{k-1} - w_i^k\|^2
	\right).\label{eqAscend}
	\end{align}
\end{lemma}

\begin{proof}
	Take any $i\in\mathcal{B}$.  Lemma~\ref{lemFinite} guarantees the finite
	termination of \coapDD \texttt{backTrack}. Now consider the backtracking
	termination condition
	\begin{align*}
	\varphi_j^+
	\geq			 
	\frac{\tilde{\rho}_{j}}{2\alpha}
	\left(
	\|\tilde{y}_{j} - w\|^2
	+
	\alpha 
	\|\hat{y}_{j} - w\|^2
	\right)+ (1-\alpha)
	\left(
	\varphi
	-
	\frac{\tilde{\rho}_{j}}{2\alpha}
	\|y - w\|^2
	\right).
	\end{align*}
	Fix some $k\geq 2$, and let $j^*\geq 1$ be the iteration at which
	\texttt{backTrack} terminates. In the above inequality, make the following
	substitutions for the internal variables of \texttt{backTrack} by those
	passed in/out of the function: $\varphi_{i,k}(z^k,x_i^k) =
	\varphi_{j^*}^+$, $\rho_i^k =
	\tilde{\rho}_{j^*}$, \myc{$\alpha_i = \alpha$}, $y_i^k = \tilde{y}_j$, $w_i^k
	= w$, $\varphi_{i,k-1}(z^k,w_i^k) = \varphi$. Furthermore, $\hat{y}_i^k =
	\hat{y}_{j^*}$ where $\hat{y}_i^k$ is defined in Definition \ref{defYhat}. Together,
	these substitutions yield~\eqref{eqAprelim}. We can then apply
	Lemma \ref{lemAv2step} to get
	\eqref{eqAscend}. 
	
	Now take any $i \in \{1,\ldots,n\} \backslash \mathcal{B}$. From
	line~\ref{lineDoNoBT} of Algorithm~\ref{AlgfullWithBT}, Assumption
	\ref{assStep}, and Lemma \ref{lemAscent0}, we directly deduce
	\eqref{eqAprelim}. Combining this relation with~\eqref{eqStep2} we obtain
	\eqref{eqAscend}.\ourqed
\end{proof}

\subsection{Finishing the Proof}\label{secProveCond1}
We now work toward establishing the conditions of Lemma~\ref{lemStrat}. Unless otherwise specified, we henceforth assume that
Algorithm~\ref{AlgfullWithBT} runs indefinitely and does not terminate at
line~\ref{lineReturn}. Termination at line~\ref{lineReturn} is dealt with in Theorem \ref{thmMain} to come. 
\begin{lemma}\label{lemFirstBiggy}
For all $i=1,\ldots,n$, we have $y_i^k-w_i^k\to 0$ and
$\varphi_k(p^k)\to 0$.
\end{lemma}
\begin{proof}Fix any $i\in\{1,\ldots,n\}$.
First, note that for all $k\geq 2$,
\begin{align}
\|y_i^{k-1}-w_i^k\|^2
&=
\|y_i^{k-1}-w_i^{k-1}\|^2
+
2\langle y_i^{k-1}-w_i^{k-1},w_i^{k-1}-w_i^k\rangle 
\coap 
+
\|w^{k-1}_i-w^k_i\|^2
\nonumber\\\nonumber
&\leq 
\|y_i^{k-1}-w_i^{k-1}\|^2
+
M_3\|w^k_i-w_i^{k-1}\|
+
\|w^k_i - w^{k-1}_i\|^2
\\\label{eqprevBound1}
&=
\|y_i^{k-1}-w_i^{k-1}\|^2
+
d_i^k,
\end{align}
where
$
d_i^k \triangleq M_3\|w_i^k-w_i^{k-1}\|
+
\|w^k_i - w^{k-1}_i\|^2
$
and $M_3\geq 0$ is a bound on $2\|y_i^{k-1}-w_i^{k-1}\|$, which must exist
because both $\{y_i^k\}$ and $\{w_i^k\}$ are bounded by lemmas
\ref{lemIsProject} and \ref{lemBounded}. Note that $d_i^k\to 0$ as a
consequence of Lemma \ref{lemIsProject}.

Second, recall Lemma \ref{lemLBPhi}, which states that there exists
$M_1,M_2\geq 0$ such that for all $k\geq 2$,
\begin{align}
\varphi_{i,k-1}(z^k,w_i^k)
&\geq 
\varphi_{i,k-1}(z^{k-1},w_i^{k-1})
-
M_1\|w_i^{k-1}-w_i^k\|
\coap 
-M_2\|G_i\|\| z^k -  z^{k-1}\|.
\label{eqNotAffine}
\end{align}
Now let, for all $k\geq 1$,
\begin{align}
\label{rdef}
r_i^k \triangleq  
\varphi_{i,k}(z^k,w_i^k)
-
\frac{\rho_i^{k+1}}{2\myc{\alpha_i}}
\|y_i^k - w_i^k\|^2,
\end{align}
so that 
\begin{align}
\sum_{i=1}^n r_i^k
=
\varphi_{k}(p^k)
-
\sum_{i=1}^n
\frac{\rho_i^{k+1}}{2\myc{\alpha_i}}
\|y_i^k - w_i^k\|^2.\label{eqSumr}
\end{align}
Using \eqref{eqprevBound1} and \eqref{eqNotAffine} in \eqref{eqAscend} yields
\begin{align}
\label{eqAscending}
(\forall k\geq 2):\quad r^k_i
\geq 
(1-\myc{\alpha_i})r^{k-1}_i 
+
e^k_i
\end{align}
where 
\begin{align}
\label{defEik}
e^k_i \triangleq  
-
(1-\myc{\alpha_i})
\left(
\frac{\rho_i^k}{2\myc{\alpha_i}}
d_i^k + 
M_1\|w_i^{k-1}-w_i^k\|
+M_2\|G_i\|\| z^k -  z^{k-1}\|
\right).
\end{align}
 
Note that $\rho_i^k$ is bounded, $0<\myc{\alpha_i}\leq 1$,
$\|G_i\|$ is finite, $\|z^{k}-z^{k-1}\|\to 0$ and $\|w_i^k - w_i^{k-1}\|\to
0$ by Lemma \ref{lemIsProject}, and $d_i^k\to 0$. Thus $e_i^k\to 0$.

Since $0<\myc{\alpha_i}\leq 1$, we may apply Lemma
\ref{lemBoundBelow} to \eqref{eqAscending} with $\tau = 1-\myc{\alpha_i}<1$, which yields
$\liminf_{k\to\infty} \{r^k_i\}\geq 0$. 
Therefore
\begin{align}
\liminf_{k\to\infty} \sum_{i=1}^n r_i^k\geq 
\sum_{i=1}^n
\liminf_{k\to\infty}  r_i^k
\geq 0.
\label{eqsumrLow}
\end{align}
On the other hand, $\limsup_{k\to\infty} \varphi_k(p^k)\leq 0$ by Lemma
\ref{lemBoundedGrad}. Therefore, using \eqref{eqSumr} and \eqref{eqsumrLow},
\begin{align*}
0 
\leq 
\liminf_{k\to\infty} 
\sum_{i=1}^n r_i^k
&=
\liminf_{k\to\infty} 
\left\{
\varphi_k(p^k)
-
\sum_{i=1}^n
\frac{\rho_i^{k+1}}{2\myc{\alpha_i}}
\|y_i^k - w_i^k\|^2
\right\}
\\
&\leq 
\liminf_{k\to\infty} \varphi_k(p^k)
\leq 
\limsup_{k\to\infty} \varphi_k(p^k)
\leq 0.
\end{align*}
Therefore $\lim_{k\to\infty}\big\{\varphi_k(p^k)\big\} = 0$. Consider any
$i\in\{1,\ldots,n\}$. Combining $\lim_{k\to\infty}\big\{\varphi_k(p^k)\big\} =
0$ with $\liminf_{k\to\infty} \sum_{i=1}^n r_i^k\geq 0,$ we have
\[
\limsup_{k\to\infty}\big\{(\rho_i^{k+1}/\myc{\alpha_i})\|y_i^k-w_i^k\|^2\big\}\leq 0
\quad \Rightarrow \quad
\rho_i^{k+1}\|y_i^k-w_i^k\|^2\to 0.
\]
Since $\rho_i^k\geq\urho_i>0$ (using Lemma \ref{lemFinite} for
$i\in\mathcal{B}$)  we conclude that
$y_i^k-w_i^k\to 0$.\ourqed
\end{proof}

We have already proved the first requirement of Lemma \ref{lemStrat}, that $y_i^k -
w_i^k\to 0$ for all $i\in\{1,\ldots,n\}$. We now work to establish the second
requirement, that $G_i z^k - x_i^k\to 0$. In the upcoming lemmas we continue
to use the quantity $\hat{y}_i^k$ which is given in Definition \ref{defYhat}.

\begin{lemma}
\label{lemYhat}
\preprintversion{Recall $\{\hat{y}_i^k\}_{k\in\nN}$ from Definition \ref{defYhat}.}
For all $i=1,\ldots,n$, $\hat{y}_i^k - w_i^k\to 0$.
\end{lemma}
\begin{proof}		
	Fix any $k\geq 1$. For all $i= 1,\ldots,n$, repeating
	\eqref{eqAprelim} from Lemma \ref{lemAscent1}, we have
	\begin{align}
	\varphi_{i,k}(z^k,w_i^k)
	&\geq 
	(1-\myc{\alpha_i})
	\left(
	\varphi_{i,k-1}(z^k,w_i^k)
	-
	\frac{\rho_i^k}{2\myc{\alpha_i}}
	\|y_i^{k-1} - w_i^k\|^2
	\right)
	\nonumber\\&\qquad 
	+
	\frac{\rho_i^k}{2\myc{\alpha_i}}\left(
	\|y_i^k - w_i^k\|^2+\myc{\alpha_i}\|\hat{y}_i^k - w_i^k\|^2
	\right)
	\nonumber\\\nonumber 
	&\geq 
	(1-\myc{\alpha_i})r_i^{k-1} + 
		\frac{\rho_i^k}{2}
	\|\hat{y}_i^k - w_i^k\|^2
	+
	e_i^k
	\end{align}
	where we have used $r_i^k$ defined \eqref{rdef} along with
	\eqref{eqprevBound1}--\eqref{eqNotAffine} and $e_i^k$ is defined in
	\eqref{defEik}. This is the same argument used in Lemma
	\ref{lemFirstBiggy}, but now we apply
	\eqref{eqprevBound1}--\eqref{eqNotAffine} to \eqref{eqAprelim}, rather
	than \eqref{eqAscend}, so that we can upper bound the $\|\hat{y}_i^k -
	w_i^k\|^2$ term. Summing over
	$i=1,\ldots,n$, yields
	\begin{align*}
	\varphi_k(p^k) = \sum_{i=1}^n\varphi_{i,k}(z^k,w_i^k)
	\geq 
	\sum_{i=1}^n 
	(1-\myc{\alpha_i}) r_i^{k-1} 
	+\sum_{i=1}^n 
	\frac{\rho_i^k}{2}\|\hat{y}_i^k - w_i^k\|^2
	+
	\sum_{i=1}^n 
	e_i^k.
	\end{align*}
	Since $\varphi_k(p^k)\to 0$, $e_i^k\to 0$, $\liminf_{k\to\infty}
	\{r_i^k\}\geq 0$, and $\rho_i^k\geq\urho_i>0$ for all $k$, the above inequality implies that $\hat{y}_i^k - w_i^k\to
	0$.\ourqed
\end{proof}

\myc{\begin{lemma}\label{lemFirstX}
For $i=1\ldots,n$, $x_i^k - x_i^{k-1}\to 0$.
\end{lemma} 
}
\begin{proof}
	 Fix $i\in\{1,\ldots,n\}$. Using the definition of $a_i^k$ 
	 in Definition \ref{defYhat}, we have for $k\geq 1$ that
	\begin{align}\nonumber 
	x_i^k + \rho_i^k a_i^k 
	   &= (1-\myc{\alpha_i})x_i^{k-1}+\myc{\alpha_i} G_i z^k - \rho_i^k(B_i x_i^{k-1}-w_i^k).
	\end{align}
	Using the definition of  $\hat{y}_i^k$, also in Definition \ref{defYhat},
	 this implies that 
	\begin{align}\label{eqUseYhat} 
&(\forall k\geq 1):&\quad x_i^k &= (1-\myc{\alpha_i})x_i^{k-1}+\myc{\alpha_i} G_i z^k - \rho_i^k(\hat{y}_i^k-w_i^k),
	\\\nonumber 
&(\forall k\geq 2):&\quad	x_i^{k-1} &= (1-\myc{\alpha_i})x_i^{k-2}+\myc{\alpha_i} G_i z^{k-1} - \rho_i^{k-1}(\hat{y}_i^{k-1}-w_i^{k-1}).
	\end{align}
	Subtracting the second of these equations from the first 
	yields, for all $k\geq 2$,
	\myc{
	\begin{align}
	x_i^k - x_i^{k-1}	
	&=
	(1-\alpha_i)(x_i^{k-1}-x_i^{k-2})
	+\alpha_i (G_i z^k -G_i z^{k-1})
	- \rho_i^k(\hat{y}_i^k-w_i^k)
	\nonumber \\
	&\qquad + \rho_i^{k-1}(\hat{y}_i^{k-1}-w_i^{k-1})	
	\nonumber 
	\end{align}
Taking norms and using the triangle inequality yields, for all $k\geq 2$, that
\begin{align}
\|x_i^k - x_i^{k-1}\|
&\leq 
\left(1-\myc{\alpha_i}\right)
\|
x_i^{k-1}- x_i^{k-2}
\|
+\tilde{e}_i^k
\label{eqUse}
\end{align}
where
$$
\tilde{e}_i^k = 
\|G_i\|\,\|z_i^k - z_i^{k-1}\|
+\rho_i^k\|\hat{y}_i^k-w_i^k\|
+
\rho_i^{k-1}
\|\hat{y}_i^{k-1}-w_i^{k-1}\|
$$	
Since $\rho_i^{k}$ is bounded from above, $\tilde{e}_i^k\to 0$ using Lemma \ref{lemYhat}, the finiteness
of $\|G_i\|$, and Lemma~\ref{lemIsProject}.
Furthermore, $\myc{\alpha_i}>0$, so we may apply
Lemma \ref{lemSeqConverge} to \eqref{eqUse} to conclude that $x_i^k -
x_i^{k-1}\to 0$.}\ourqed
\end{proof}

\begin{lemma}\label{lemSecondX}
	For $i=1,\dots,n$,
$G_i z^k - x_i^k\to 0$.
\end{lemma}

\begin{proof}
Recalling \eqref{eqUseYhat}, we first write
\begin{align}
x_i^k 
&=
(1-\myc{\alpha_i})x_i^{k-1}+\myc{\alpha_i} G_i z^k -\rho_i^k(\hat{y}_i^k - w_i^k) \nonumber \\
\Leftrightarrow \qquad
\myc{\alpha_i}\left(G_i z^k - x_i^k\right)
&=
(1-\myc{\alpha_i})(x_i^k - x_i^{k-1})
+
\rho_i^k(\hat{y}_i^k-w_i^k).
\label{eqAnAnalog}
\end{align}
Lemma~\ref{lemFirstX} implies that the first term
on the right-hand side of~\eqref{eqAnAnalog} converges to zero.  Since
$\{\rho_i^k\}$ is bounded, Lemma~\ref{lemYhat} implies that the second term
on the right-hand side also converges to zero.  \myc{Since $\myc{\alpha_i}>0$,
we conclude that} $\|G_i z^k-x_i^k\|\to 0$. \ourqed
\end{proof}
\noindent Finally, we can state our convergence result for 
Algorithm~\ref{AlgfullWithBT}: 
\begin{theorem}\label{thmMain}
	Suppose that assumptions \ref{AssMonoProb}-\ref{assStep} hold. 
	If Algorithm~\ref{AlgfullWithBT} terminates by reaching line~\ref{lineReturn}, 
	then its 
	final iterate is a member of the extended solution set $\calS$.
	Otherwise, the
	sequence $\{(z^k,\bw^k)\}$ generated by Algorithm~\ref{AlgfullWithBT}
	converges weakly to some point $(\bar z,\overline{\bw})$ in the extended
	solution set $\calS$ of~\eqref{prob1} defined
	in~\eqref{defCompExtSol}. Furthermore, $x_i^k\rightharpoonup G_i \bar z$
	and $y_i^k\rightharpoonup \overline{w}_i$ for all $i=1,\ldots,n-1$,
	$x_{n}^k\rightharpoonup \bar z$, and $y_n^k\rightharpoonup
	-\sum_{i=1}^{n-1}G_i^* \overline{w}_i$. 
\end{theorem}
\begin{proof}
	For the finite termination result we refer to Lemma 5 of \cite{johnstone2018projective}.
	Otherwise, lemmas~\ref{lemFirstBiggy} and~\ref{lemSecondX} imply that 
	the hypotheses of Lemma~\ref{lemStrat}, hold, and the result follows.\ourqed
\end{proof}

%% file: portfolio.tex
\newcommand{\onef}{\texttt{ps1fbt}}
\newcommand{\twof}{\texttt{ps2fbt}}
\newcommand{\ada}{\texttt{ada3op}}
\newcommand{\stoch}{\texttt{stoc3op}}
\newcommand{\cp}{\texttt{cp-bt}}
\newcommand{\tseng}{\texttt{tseng-pd}}
\newcommand{\frb}{\texttt{frb-pd}}

\section{Numerical Experiments}\label{secNum}
All our numerical experiments were implemented in Python (using
\texttt{numpy} and \texttt{scipy}) on an Intel Xeon workstation running Linux with 16 cores and 64 GB of RAM. \myc{The code is available via github at \url{https://github.com/projective-splitting/coco}}. 
\preprintversion{We restricted our
attention to algorithms with comparable features and benefits to
our proposed method.  Thus we only considered methods that:
\begin{enumerate}
\item \label{algprop:1fullsplit} Are first-order and ``fully split" the
      problem (that is, separate the linear operators $G_i$ from the resolvent
      calculations, and use gradient-type steps for smooth functions),
\item Do not (either approximately or exactly) solve a linear system of equations
      at each iteration or before the first iteration,
\item Avoid having to apply ``smoothing" to nonsmooth operators,
\item \label{algprop:backtrack} Incorporate a backtracking linesearch in a
      manner that avoids the need for bounds on Lipschitz or cocoercivity
      constants, and
\item \label{algprop:exactJ} Do not use iterative approximation of
      resolvents.
\end{enumerate}
The last property we include for reasons of simplicity, while the rest
contribute to making algorithms scalable and easy to apply. For a given
application, there may of course be effective algorithms which could have been
considered but do not satisfy all of the above requirements.  However, because
of the general desirability of properties
\ref{algprop:1fullsplit}-\ref{algprop:backtrack} and the relative simplicity
of algorithms with property~\ref{algprop:exactJ}, we only considered methods
having all of them.

}
We compared this paper's backtracking one-forward-step projective splitting
algorithm given in Algorithm~\ref{AlgfullWithBT} (which we call
\texttt{ps1fbt}) with the following \preprintversion{methods:}
\coapversion{methods, selected for their similarities in features (especially
the ability to ``fully split'' problems and having deterministic convergence 
guarantees), applicability, and implementation
effort:}
\begin{itemize}
	\item The two-forward-step projective splitting algorithm with
	backtracking we developed in \cite{johnstone2018projective}
	(\texttt{ps2fbt}).  This method requires only Lipschitz continuity of
	single-valued operators, as opposed to cocoercivity.
	\item The adaptive three-operator splitting algorithm of
	\cite{pedregosa2018adaptive} (\texttt{ada3op}) (where ``adaptive" is used
	to mean ``backtracking linesearch"); this method is a backtracking
	adaptation of the fixed-stepsize method proposed in~\cite{davis2015three}.
	This method requires $G_i = I$ in problem~\eqref{prob1} and hence can only
	be readily applied to two of the three test applications described below.
	\item The backtracking linesearch variant of the Chambolle-Pock primal-dual
	splitting method \cite{malitsky2018first} (\cp).
	\item The algorithm of \cite{combettes2012primal}. This is essentially Tseng's method applied to a product-space
	``monotone + skew" inclusion in the following way: Assume $T_n$ is Lipschitz monotone, problem
	\eqref{prob2} is equivalent to finding $p\triangleq
	(z,w_1,\ldots,w_{n-1})$ such that $w_i\in T_i G_i z$ (which is equivalent
	to $G_iz\in T_i^{-1}w_i$) for $i=1,\ldots,n-1$, and $\sum_{i=1}^{n-1}G_i^* w_i=-T_n z$. In other words, we wish to solve $0\in \tilde{A}p +
	\tilde{B} p$, where $\tilde{A}$ and $\tilde{B}$ are defined by
\begin{align}\label{eqMonPlSk1}
\tilde{A}p &= \{0\} \times T_1^{-1}w_1 \times \cdots \times T_{n-1}^{-1}w_{n-1}
\\ \label{eqMonPlSk2}
\tilde{B}p &=
\left[
\begin{array}{c}
T_n z\\
0\\
\vdots\\
0
\end{array}
\right]
+
\left[
\begin{array}{ccccc}
0&G_1^*&G_2^*&\hdots &G_{n-1}^*\\
-G_1&0& \hdots & \hdots & 0\\
\vdots &\vdots & \ddots & \ddots& \vdots\\
-G_{n-1} & 0 & \hdots & \hdots & 0
\end{array}
\right]
\left[
\begin{array}{c}
z\\
w_1\\
\vdots\\
w_{n-1}
\end{array}
\right].
\end{align}
$\tilde{A}$ is maximal monotone, while $\tilde{B}$ is the sum of two
Lipshitz monotone operators (the second being skew linear), and therefore also
Lipschitz monotone. The algorithm in \cite{combettes2012primal} is
essentially Tseng's forward-backward-forward method~\cite{tseng2000modified}
applied to this inclusion, using resolvent steps for $\tilde{A}$ and forward
steps for $\tilde{B}$. Thus, we call this method \tseng. In order to achieve
good performance with \tseng\space we had to incorporate a diagonal
preconditioner as proposed in \cite{vu2013variable}.
\item The recently proposed forward-reflected-backward
method~\cite{tam2018forward}, applied to this same primal-dual inclusion $0\in
\tilde{A}p + \tilde{B} p$ specified by
\eqref{eqMonPlSk1}-\eqref{eqMonPlSk2}.  We call this method \frb. 
\end{itemize}

\preprintversion{Recently there have been several stochastic 
extensions of \ada\space and
\cp\space \cite{yurtsever2016stochastic,zhao2018stochastic,pedregosa2019proximal}. The
method of \cite{zhao2018stochastic} requires estimates of the Lipschitz
constants and matrix norms, and so does not satisfy our experimental
requirements. Since one of our problems is not in ``finite-sum" format, and
another includes a matrix $G_i$ which is not equal to the identity, the
methods of \cite{yurtsever2016stochastic,pedregosa2019proximal} could only be
applied to one of our three test problems. Even for this problem, the number
of training examples in the two datasets were $60$ and $127$, respectively,
while the feature dimensions were $7,\!705$ and $19,\!806$, so finite-sum
methods are not particularly suitable. For these reasons we did not include
these methods in our experiments.}


\subsection{Portfolio Selection}\label{secPort}
Consider the optimization problem:
\begin{align} \label{probPort}
\min_{x\in\rR^d}
F(x)\triangleq 
x^\top Q x
\quad \text{s.t.}\quad 
m^\top x \geq r,\quad 
\sum_{i=1}^d x_i = 1,x_i\geq 0,
\end{align}
where $Q\succeq 0$, $r>0$, and $m\in\rR^d_+$. This model arises
in Markowitz portfolio theory. We chose this particular problem because it
features two constraint sets (a general halfspace and a simplex) onto which it
is easy to project individually, but whose intersection poses a more difficult
projection problem.  This property makes it difficult to apply first-order
methods such as ISTA/FISTA \cite{beck2009fast} as they can only perform one
projection per iteration and thus cannot fully split the problem. On the other
hand, projective splitting can handle an arbitrary number of constraint sets
so long as one can compute projections onto each of them.  We consider a
fairly large instance of this problem so that standard interior point methods
(for example, those in the CVXPY~\cite{cvxpy} package) are disadvantaged by
their high per-iteration complexity and thus not generally competitive with
first-order methods. Furthermore, backtracking variants of first-order methods are
preferable for large problems as they avoid the need to estimate the largest
eigenvalue of $Q$.

To convert \eqref{probPort} to a monotone inclusion, we set $A_1 = N_{C_1}$
where $N_{C_1}$ is the normal cone of the simplex $C_1 =
\{x\in\rR^d:\sum_{i=1}^d x_i=0,x_i\geq 0\}$. We set $B_1 = 2Qx$, which is the
gradient of the objective function and is cocoercive (and
Lipschitz-continuous). Finally, we set $A_2 = N_{C_2}$, where $C_2 =
\{x:m^\top x\geq r\}$, and let $B_2$ be the zero operator. Note that the
resolvents of $N_{C_1}$ and $N_{C_2}$ (that is, the projections onto $C_1$ and
$C_2$) are easily computed in $\bigO(d)$ operations~\cite{michelot1986finite}.
With this notation, one may write~\eqref{probPort} as the the problem of
finding $z\in\rR^d$ such that
\begin{align}\nonumber 
0\in A_1 z + B_1 z + A_2 z,
\end{align} 
which is an instance of \eqref{prob1} with $n=2$ and $G_1=G_2=I$. 

To terminate each method in our comparisons, we used the following common
criterion incorporating both the objective function and the constraints
of~\eqref{probPort}:
\begin{align}
c(x) &\triangleq 
    \max\left\{\frac{F(x)-F^*}{F^*},0\right\}
    - \min\{m^\top x-r,0\}
    + \left|\sum_{i=1}^d x_i - 1\right|
    \coap 
    - \max\{0,\min_i x_i\},
    \label{eq_ce}
\end{align}
where $F^*$ is the optimal value of the problem. Note that $c(x)=0$ if and
only if $x$ solves \eqref{probPort}. To estimate $F^*$, we used the best feasible
value returned by any method after $1000$ iterations.  


We generated random instances of \eqref{probPort} as follows:  we
set $d=10,000$ to obtain a relatively large instance of the problem. We then
generated a $d\times d$ matrix $Q_0$ with each entry drawn from $\calN(0,1)$.
The matrix $Q$ is then formed as $(1/d)\cdot Q_0 Q_0^\top$, which is
guaranteed to be positive semidefinite. We then generate the vector
$m\in\rR^d$ of length $d$ to have entries uniformly distributed between $0$
and $100$. The constant $r$ is set to $\delta_r\sum_{i=1}^d m_i/d$ for various
values of $\delta_r>0$. We solved the problem for
$\delta_r\in\{0.5,0.8,1,1.5\}$.

All methods were initialized at the same point $[1~1~\ldots~1]^\top/d$. For
all the backtracking linesearch procedures except \cp\space, the initial
stepsize estimate is the previously discovered stepsize; at the first
iteration, the initial stepsize is $1$. For \cp\space we allowed the stepsize
to increase in accordance with \cite[Algorithm 4]{malitsky2018first}, as
performance was poor otherwise. The backtracking stepsize decrement factor
($\delta$ in Algorithm \ref{AlgBackTrack}) was $0.7$ for all algorithms.

For \onef\space and \twof, $\rho_1^k$ was discovered via backtracking. We also
set the other stepsize $\rho_2^k$ equal to $\rho_1^k$ at each iteration. While
this is not necessary, this heuristic performed well and eliminated $\rho_2^k$
as a separately tunable parameter. For the averaging parameters in
\texttt{ps1fbt}, we used $\alpha_1=0.1$ and $\alpha_2=1$ (which is possible because $L_2=0$).  For \texttt{ps1fbt} we set
$\hat{\theta}_1 = x_1^0$ and $\hat{w}_1 =2 Q x_1^0$.

\begin{table}
	\centering
	\begin{tabular}{c||c|c|c|c|c}
		&\multicolumn{4}{c}{$\delta_r$}
		\\
		&$0.5$&$0.8$&$1$ & $1.5$
		\\
		\hline	
		\hline	
		\onef\,($\gamma$)
		& $0.01$ & $0.01$ & $0.5$ & $5$
		\\
		\hline 
		\twof\,($\gamma$) & $0.1$ & $0.1$ & $10$  & $10$
		\\
		\hline 
		\cp\, ($\beta^{-1}$) & $1$ & $1$ & $2$ & $2$
		\\
		\hline 
		\tseng\, ($\gamma_{pd}$) & $1$ & $1$ & $1$  & $10$
		\\
		\hline 
		\frb\, ($\gamma_{pd}$) & $1$ & $1$ & $10$  & $10$
		\\
		\hline
		\hline  
	\end{tabular}
	\caption{Tuning parameters for the portfolio problem (\ada\space does not have a tuning parameter.)}
	\label{table-tune}
\end{table}

For \tseng\space and \frb, we used the following preconditioner:
\begin{align}\label{eq-precond}
U = \text{diag}(I_{d\times d},\gamma_{pd} I_{d\times d},\gamma_{pd} I_{d\times d})
\end{align}
where $U$ is used as in \cite[Eq.~(3.2)]{vu2013variable} for \tseng\space
($M^{-1}$ on \cite[p.~7]{tam2018forward} for \frb). In this case, the
``monotone + skew" primal-dual inclusion described in
\eqref{eqMonPlSk1}-\eqref{eqMonPlSk2} features two $d$-dimensional dual
variables in addition to the $d$-dimensional primal variable. The parameter
$\gamma_{pd}$ changes the relative size of the steps taken in the primal and
dual spaces, and plays a similar role to $\gamma$ in our algorithm (see
Algorithm \ref{AlgProjUpdate}). The parameter $\beta$ in
\cite[Algorithm~4]{malitsky2018first} plays a similar role for \cp. For all of
these methods, we have found that performance is highly sensitive to this
parameter: the primal and dual stepsizes need to be balanced.  The only method
not requiring such tuning is \ada, which is a purely primal method. With this
setup, all the methods have one tuning parameter except
\ada\space, which has none.  For each method, we manually tuned the parameter
for each $\delta_r$; Table~\ref{table-tune} shows the final choices.

We calculated the criterion $c(x)$ in \eqref{eq_ce} for $x_1^k$ computed by
\texttt{ps1fbt} and \texttt{ps2fbt}, $x_t$ computed on Line 3 of
\cite[Algorithm 1]{pedregosa2018adaptive} for \texttt{ada3op}, $y^{k}$
computed in  \cite[Algorithm 4]{malitsky2018first} for \cp, and the primal
iterate for \tseng\space and \frb.  Table~\ref{table-results} displays the
average number iterations and running time, over $10$ random trials, until
$c(x)$ falls (and stays) below $10^{-5}$ for each method. Examining the table,
\begin{itemize}
	\item For all four problems, \onef\space outperforms \twof. This behavior is not
	suprising, as \onef\space only requires one forward step per iteration,
	rather than two. Since the matrix $Q$ is large and dense, reducing the
	number of forward steps should have a sizable impact.
	\item For $\delta_r<1$, \onef\space is the best-performing method.
	However, for $\delta_r\geq 1$, \ada\space is the quickest.
\end{itemize}

\begin{table}
	\centering
	\begin{tabular}{c||c|c|c|c}
		&\multicolumn{4}{c}{$\delta_r$}
		\\
		&$0.5$&$0.8$&$1$&$1.5$\\ 
		\hline 
		\hline
		\onef & \textbf{3.6} (102) & \textbf{4.7} (102) & 16.3 (583) & 8.5 (255.2) \\
		\hline 
		\twof
		& 5.0 (151.1) & 7.9 (155) & 24.3 (523.4) & 9.2 (222.9) \\
		\hline 
		\ada & 5.3 (180.8) &9.2 (180.8) & \textbf{6.8} (174.3) & \textbf{3.4} (89.2) \\
		\hline 
		\cp& 6.2 (136) & 8.3 (134.3) & 11.8 (218.4) & 5.6 (113.6) \\
		\hline 
		\tseng
		&15.9 (387.1) &21 (387.8) & 25.7 (525.3) & 11.1 (245.4) \\
		\hline 
		\frb &10.5 (559.9) &16.4 (560.4) & 22.8 (1074.8) & 6.3 (350.8) \\
		\hline
		\hline  
	\end{tabular}
	\caption{For the portfolio problem, average running times in seconds and
	iterations (in parentheses) for each method until $c(x)<10^{-5}$ for all
	subsequent iterations across 10 trials. The best time in each column is in
	\textbf{bold}.}
	\label{table-results}
\end{table}



%% file: group-lasso.tex
\newcommand{\calG}{\mathcal{G}}
\subsection{Sparse Group Logistic Regression}
Consider the following problem:
\begin{align}
\min_{\substack{x_0\in\rR \\ x\in\rR^d}}
\left\{
\sum_{i=1}^n
\log\!\Big(1+\exp\!\big(\!-y_i(x_0 + a_i^\top x)\big)\Big)
+
\lambda_1\|x\|_1
+
\lambda_2
\sum_{g\in\mathcal{G}}
\|x_g\|_2\right\},
\label{eq-group-lr}
\end{align}
where $a_i\in\rR^d$ and $y_i\in\{\pm 1\}$ for $i=1,\ldots,n$ are given data,
$\lambda_1,\lambda_2\geq 0$ are regularization parameters, and $\calG$ is a
set of subsets of $\{1,\ldots,d\}$ such that no element is in more than one
group $g\in\calG$. This is the non-overlapping group-sparse logistic
regression problem, which has applications in bioinformatics, image
processing, and statistics~\cite{simon2013sparse}. It is well understood that
the $\ell_1$ penalty encourages sparsity in the solution vector. On the other
hand the group-sparse penalty encourages \emph{group sparsity}, meaning that
as $\lambda_2$ increases more groups in the solution will be set entirely to
$0$. The group-sparse penalty can be used when the features/predictors can be
put into correlated groups in a meaningful way. As with the portfolio
experiment, this problem features two nonsmooth regularizers and so methods
like FISTA cannot easily be applied.

Problem \eqref{eq-group-lr} may be treated as a special case of
\eqref{ProbOpt} with $n=2$, $G_1=G_2=I$, and
\begin{align*}
h_1(x_0,x) &= \sum_{i=1}^n \log\!\Big(1+\exp\!\big(-y_i(x_0 + a_i^\top x)\big)\Big) &
h_2(x_0,x) &= 0 \\
f_1(x_0,x) &= \lambda_1\|x\|_1 &
f_2(x_0,x) &= \lambda_2 \sum_{g\in\mathcal{G}} \|x_g\|_2.
\end{align*}
Since the logistic regression loss has a Lipschitz-continuous gradient and the
$\ell_1$-norm and non-overlapping group-lasso penalties both have
computationally simple proximal operators, all our candidate methods may be
applied.

We applied~\eqref{eq-group-lr} to two bioinformatics classification problems
with real data. Following~\cite{simon2013sparse}, we use the breast cancer
dataset of \cite{ma2004two} and the inflammatory bowel disease (IBD) dataset
of~\cite{burczynski2006molecular}.\footnote{The breast cancer dataset is
available at \url{https://www.ncbi.nlm.nih.gov/geo/query/acc.cgi?acc=GSE1379}.
The IBD dataset is available at
\url{https://www.ncbi.nlm.nih.gov/geo/query/acc.cgi?acc=GSE3365}.} The breast
cancer dataset contains gene expression levels
for 60 patients with estrogen-positive breast
cancer. The patients were treated with tamoxifen for 5 years and classified
based on whether the cancer recurred (there were 28 recurrences). The goal is
to use the gene expression values to predict recurrence.
The IBD data set contains gene expression levels
for 127 patients, 85 of which have IBD. The IBD data set
actually features three classes: ulcerative  colitis (UC), Crohn's disease
(CD), and normal, and so the most natural goal would be to perform
three-way classification. For simplicity, we considered a two-way
classification problem of UC/CD patients versus normal patients.

For both datasets, as in \cite{simon2013sparse}, the group structure $\calG$
was extracted from the C1 dataset~\cite{subramanian2005gene}, which groups
genes based on cytogenetic position data.\footnote{The C1 dataset is available at
\url{http://software.broadinstitute.org/gsea/index.jsp}.} Genes that are in
multiple C1 groups were removed from the dataset.\footnote{Overlapping
group norms can also be handled with our method, but using a different problem
formulation than \eqref{eq-group-lr}.} We also removed genes that could not be
found in the C1 dataset, although doing so was not strictly necessary. After
these steps, the breast cancer data had 7,705 genes in 324 groups, with
each group having an average of 23.8 genes.  For the IBD data there were
19,836 genes in 325 groups, with an average of 61.0 genes per group. Let $A$ be
the data matrix with each row is equal to $a_i^\top\in\rR^{d}$ for
$i=1,\ldots,n$; as a final preprocessing step, we normalized the columns of
$A$ to have unit $\ell_2$-norm, which tended to improve the performance of the
first-order methods, especially the primal-dual ones.

For simplicity
we set the regularization
parameters to be equal: $\lambda_1=\lambda_2\triangleq\lambda$.  In practice,
one would typically solve \eqref{eq-group-lr} for various values of $\lambda$
and then choose the final model based on cross-validation performance combined
with other criteria such as sparsity. Therefore, to give an overall sense of
the performance of each algorithm, we solved \eqref{eq-group-lr} for three
values of $\lambda$: large, medium, and small, corresponding to decreasing the
amount of regularization and moving from a relatively sparse solution to a
dense solution. For the breast cancer data, we selected
$\lambda\in\{0.05,0.5,0.85\}$ and for IBD we chose $\lambda\in \{0.1,0.5,1\}$.
The corresponding number of non-zero entries, non-zero groups, and training
error of the solution are reported in Table~\ref{table-train}. Since the goal
of these experiments is to assess the computational performance of the
optimization solvers, we did not break up the data into training and test
sets, instead treating the entire dataset as training data.

\begin{table}
	\centering
	\begin{tabular}{c||c|c|c||c|c|c|}
	&\multicolumn{3}{c||}{$\lambda$ (breast cancer)}&\multicolumn{3}{c|}{$\lambda$ (IBD)}\\
	 & 0.05&0.5&0.85&0.1&0.5&1.0\\
	\hline
	\# Nonzeros & 114 & 50 & 20&135&40&18\\
	\# Nonzero groups & 16 & 7 & 3&13&4&2\\
	Training error &0\%&5\%&35\%&0\%&5.5\%&26.8\%\\
	\hline
	\end{tabular}
\caption{The number of nonzeros and nonzero groups in the solution, along with the
training error, for each value of $\lambda$.}
\label{table-train}
\end{table}

We initialized all the methods to the $0$ vector. As in the portfolio problem,
all stepsizes were initially set to $1$. Since the logistic regression function
does not have uniform curvature, we allowed the initial trial stepsize in the
backtracking linesearch to increase by a factor of $1.1$ multiplied by the
previously discovered stepsize. The methods \onef, \cp, and \ada\space have an
upper bound on the trial stepsize at each iteration, so the trial stepsize was
taken to be the minimum of $1.1$ multiplied by the previous stepsize and this
upper bound.

Otherwise, the setup was the same as the portfolio experiment. \tseng\space
and \frb\space use the same preconditioner as given in \eqref{eq-precond}. For
\onef\space and \twof\space we set $\rho_2^k$ to be equal to the discovered
backtracked stepsize $\rho_1^k$ at each iteration.  For \texttt{ps1fbt} we
again set $\hat{\theta}_1 = x_1^0$, $\hat{w}_1 =\nabla h_1(x_1^0)$, and
$\alpha_1^k$ fixed to $0.1$. As such, all methods (except \ada) have one
tuning parameter which was hand-picked for each method; the chosen values
are given in Table \ref{table-tune2}.

\begin{table}
	\centering
	\begin{tabular}{c||c|c|c||c|c|c|c|}
		&\multicolumn{3}{c||}{$\lambda$ (breast cancer)}
		&
		\multicolumn{3}{c|}{$\lambda$ (IBD)}
		\\
		&0.05&0.5&0.85&0.1&0.5&1.0
		\\
		\hline
		\hline
		\onef\,($\gamma$)
		& $0.05$ & $10^2$ & $10^2$ &0.1&1&1
		\\
		\hline
		\twof\,($\gamma$) & $1$ & $10^2$ & $10^5$  &1&1&1
		\\
		\hline
		\cp\, ($\beta^{-1}$) & $10$ & $10^3$ & $10^4$ &$10^4$&$10^3$&$10^5$
		\\
		\hline
		\tseng\, ($\gamma_{pd}$) & $10^3$ & $10^5$ & $10^5$  &$10^4$&$10^6$&$10^6$
		\\
		\hline
		\frb\, ($\gamma_{pd}$) & $10^3$ & $10^5$ & $10^5$  &$10^4$&$10^6$&$10^6$
		\\
		\hline
		\hline
	\end{tabular}
	\caption{Tuning parameters for sparse group LR (\ada\space does not have a tuning parameter).}
	\label{table-tune2}
\end{table}

\begin{figure}[!htb]
	\minipage{0.32\textwidth}
	\includegraphics[width=\linewidth]{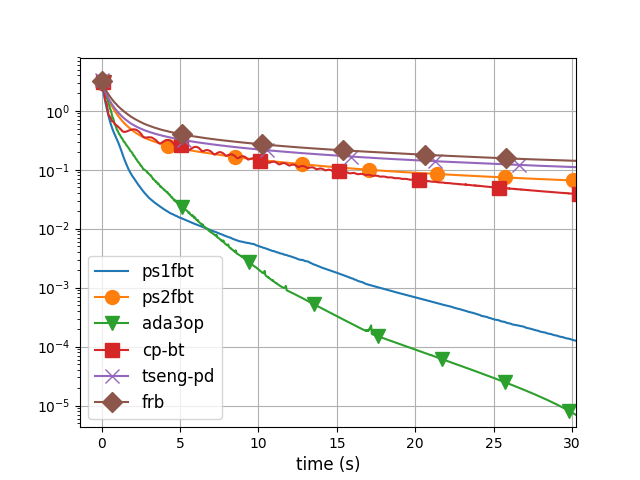}
	\endminipage\hfill
	\minipage{0.32\textwidth}
	\includegraphics[width=\linewidth]{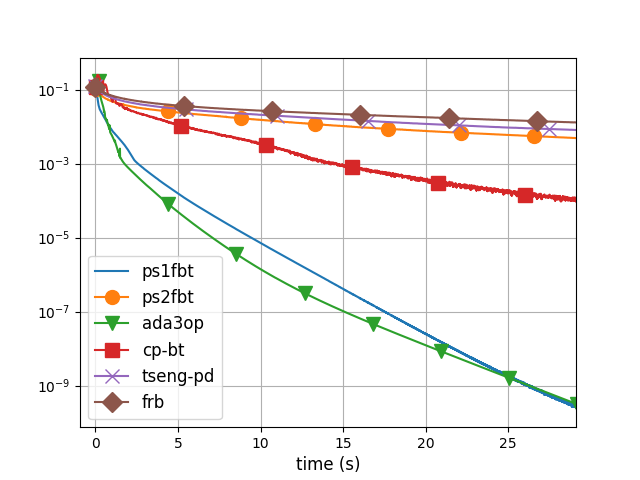}
	\endminipage\hfill
	\minipage{0.32\textwidth}
	\includegraphics[width=\linewidth]{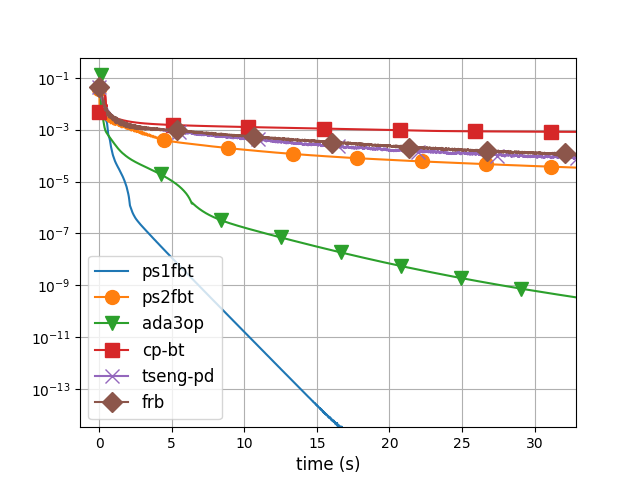}
	\endminipage\hfill
	\minipage{0.32\textwidth}
	\includegraphics[width=\linewidth]{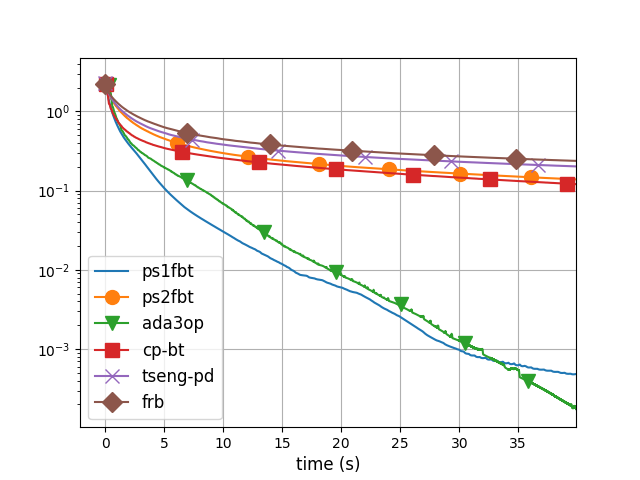}
	\endminipage\hfill
	\minipage{0.32\textwidth}
	\includegraphics[width=\linewidth]{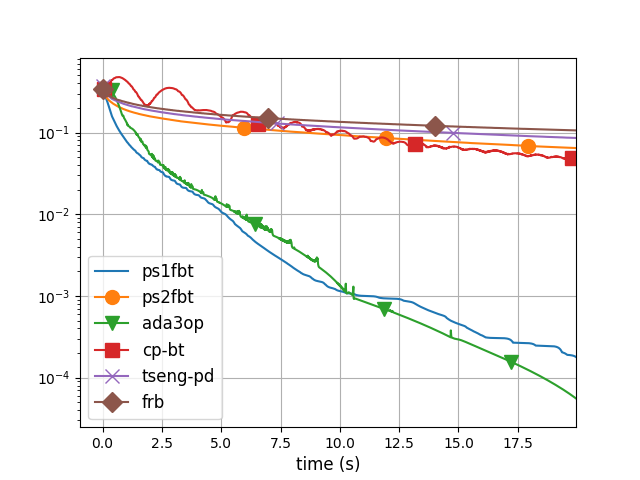}
	\endminipage\hfill
	\minipage{0.32\textwidth}%
	\includegraphics[width=\linewidth]{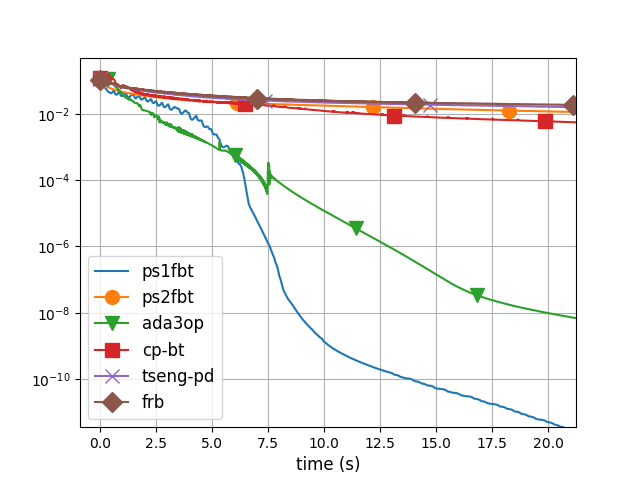}
	\endminipage
	\caption{Results for \eqref{eq-group-lr} applied to bioinformatics
	classification problems. The top row shows breast cancer data with left:
	$\lambda=0.05$; middle: $\lambda=0.5$; right: $\lambda=0.85$. The bottom row
	shows IBD data with left: $\lambda=0.1$; middle: $\lambda=0.5$; right:
	$\lambda=1.0$.  The $y$-axis is relative objective error: $\big(F(x_0,x)-F^*\big)/F^*$
	and the $x$-axis is elapsed running time in seconds.}
	\label{Fig-res1}
\end{figure}


Figure~\ref{Fig-res1} shows the results of the experiments, plotting
$(F(x_0,x)-F^*)/F^*$ against time for each algorithm, where $F$ is the
objective function in~\eqref{eq-group-lr} and $F^*$ is the estimated optimal
value. To approximate $F^*$, we ran each algorithm for 4,000 iterations and
took the lowest value obtained. Overall, \onef\space and \ada\space were much
faster than the other methods. For the highly regularized cases (the right
column of the figure), \onef\space was faster than all other methods. For
middle and low regularization, \onef\space and \ada\space are comparable, and
for $\lambda=0.05$ \ada\space is slightly faster for the the breast cancer
data. The methods \onef\space and \ada\space may be succesful because they
exploit the cocoercivity of the gradient, while \twof, \tseng,and \frb\space
only treat it as Lipschitz continuous. \cp\space also exploits cocoercivity,
but its convergence was slow nonetheless. We discuss the performance of \onef\space versus \twof\space more in Section \ref{sec1v2}.

%% file: onefVtwof.tex
\subsection{Final Comments: \onef\space versus \twof}\label{sec1v2}
On the portfolio problem, \onef\space and \twof\space have
fairly comparable performance, with \onef\space being slightly faster.
However, for the group logistic regression problem, \onef\space is
significantly faster. Given that both methods are based on the same projective
splitting framework but use different forward-step procedures to
update $(x_1^k,y_1^k)$, this difference may be somewhat surprising. Since
\onef\space only requires one forward step per iteration while \twof\space
requires two, one might expect
\onef\space to be about twice as fast as \twof. But for the group logistic
regression problem, \onef\space significantly outpaces this level of performance.

Examining the stepsizes returned by backtracking for both methods reveals that
\onef\space returns much larger stepsizes for the logistic regression problem,
typically $2$-$3$ orders of magnitude larger; see Figure~\ref{Fig-steps}. For
the portfolio problem, where the performance of
the two methods is more similar, this is not the case: the \onef\space
stepsizes are typically about twice as large as the \twof\space stepsizes, in
keeping with their theoretical upper bounds of $1/L_i$ and $2(1-\alpha_i)/L_i$,
respectively.

Note that the portfolio problem has a smooth function
which is quadratic and hence has the same curvature everywhere, while
group logisitic regression does not. We hypothesize that the backtracking
scheme in \onef\space does a better job adapting to nonuniform
curvature. A possible reason for this behavior is that the termination criterion for
the backtracking search in \onef\space may be weaker than for \twof. For
example, while \twof\space requires $\varphi_{i,k}$ to be positive at each
iteration $k$ and operator $i$, \onef\space does not.

\begin{figure}
	\minipage{0.49\textwidth}
	\includegraphics[width=\linewidth]{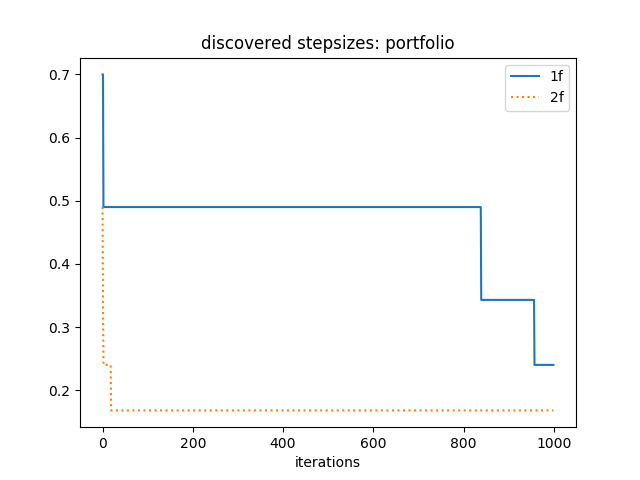}	
	\endminipage\hfill
	\minipage{0.49\textwidth}
	\includegraphics[width=\linewidth]{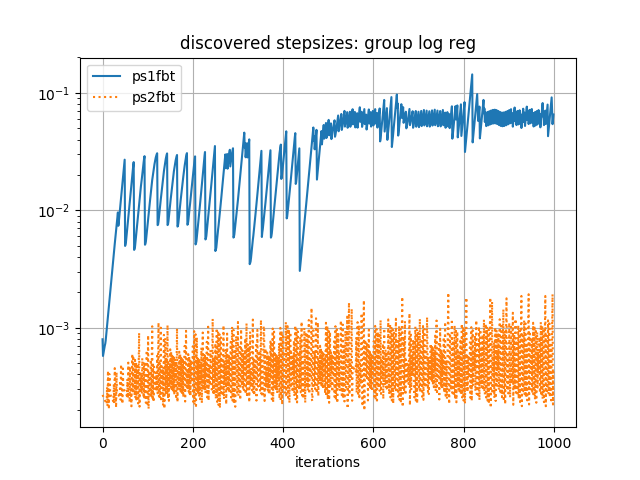}	
	\endminipage\hfill
	\caption{Discovered backtracking stepsizes for \onef\space and \twof\space. Left: portfolio problem with $\delta_{r}=0.5$. Right: group logistic regression problem applied to the IBD data with $\lambda=1$.}
	\label{Fig-steps}
\end{figure}